\documentclass[a4paper, 11pt]{article}
\usepackage[a4paper,hmargin={2.54cm,2.54cm},vmargin={3.17cm,3.17cm}]{geometry}

\usepackage{amsmath,amssymb}
\usepackage{amsthm}
\usepackage{amssymb}

\newcommand{\sgn}{\mathrm{sgn}}
\newcommand{\const}{\mathrm{const}}
\newcommand{\inters}{\mathrm{int}}
\newcommand{\Tr}{\mathrm{Tr}}
\newcommand{\Arg}{\mathrm{Arg}}
\newcommand{\definition}{\mathrm{def}}

\renewcommand{\Im}{\mathrm{Im}}
\renewcommand{\Re}{\mathrm{Re}}

\newtheorem{lemma}{Lemma}
\newtheorem{theorem}{Theorem}

\numberwithin{equation}{section}
\numberwithin{lemma}{section}
\numberwithin{theorem}{section}
\numberwithin{statement}{section}

\begin{document}

\begin{center}
\Large \textbf{ A large time asymptotics for the solution of the Cauchy problem for the Novikov-Veselov equation at negative energy with non-singular scattering data }
\end{center}

\begin{center}
A.V. Kazeykina \footnote{CMAP, Ecole Polytechnique, Palaiseau, 91128, France; email: kazeykina@cmap.polytechnique.fr}
\end{center}

\textbf{Abstract.} In the present paper we are concerned with the Novikov--Veselov equation at negative energy, i.e. with the $ ( 2 + 1 ) $--dimensional analog of the KdV equation integrable by the method of inverse scattering for the two--dimensional Schr\"odinger equation at negative energy. We show that the solution of the Cauchy problem for this equation with non--singular scattering data behaves asymptotically as $ \frac{ \const }{ t^{ 3/4 } } $ in the uniform norm at large times $ t $. We also prove that this asymptotics is optimal.

\section{Introduction}
In the present paper we consider the Novikov--Veselov equation
\begin{gather}
\notag
\partial_t v = 4 \Re ( 4 \partial_z^3 v + \partial_z( v w ) - E \partial_z w ),\\
\label{NV}
\partial_{ \bar z } w = - 3 \partial_z v, \quad v = \bar v, \quad E \in \mathbb{R}, \quad E < 0, \\
\notag
v = v( x, t ), \quad w = w( x, t ), \quad x = ( x_1, x_2 ) \in \mathbb{R}^2, \quad t \in \mathbb{R},
\end{gather}
where
\begin{equation*}
\partial_t = \frac{ \partial }{ \partial t }, \quad \partial_z = \frac{ 1 }{ 2 } \left( \frac{ \partial }{ \partial x_1 } - i \frac{ \partial }{ \partial x_2 } \right), \quad \partial_{ \bar z } = \frac{ 1 }{ 2 } \left( \frac{ \partial }{ \partial x_1 } + i \frac{ \partial }{ \partial x_2 } \right).
\end{equation*}
We will say that $ ( v, w ) $ is a rapidly decaying solution of (\ref{NV}) if
\begin{subequations}
 \label{sol_definition}
\begin{align}
 & \bullet v, w \in C( \mathbb{R}^2 \times \mathbb{R} ), \quad v( \cdot, t ) \in C^3( \mathbb{R}^3 ), \\
 & \bullet | \partial_{ x }^{ j } v( x, t ) | \leqslant \frac{ q( t ) }{ ( 1 + | x | )^{ 2 + \varepsilon } }, \; | j | \leqslant 3, \text{ for some $ \varepsilon > 0 $,} \quad w( x, t ) \to 0, | x | \to \infty, \\
 & \bullet ( v, w ) \text{ satisfies (\ref{NV}).}
\end{align}

\end{subequations}

Note that if $ v( x, t ) = v( x_1, t ) $, $ w( x, t ) = w( x_1, t ) $, then (\ref{NV}) is reduced to the classic KdV equation. In addition, (\ref{NV}) is integrable via the inverse scattering method for the two--dimensional Schr\"odinger equation
\begin{equation}
\label{schrodinger}
 L \psi = E \psi, \quad L = - \Delta + v( x ), \quad x = ( x_1, x_2 ), \quad E = E_{ fixed }.
\end{equation}
In this connection, it was shown (see \cite{M}, \cite{NV1}, \cite{NV2}) that for the Schr\"odinger operator $ L $ from (\ref{schrodinger}) there exist appropriate operators $ A $, $ B $ (Manakov L--A--B triple) such that (\ref{NV}) is equivalent to
\begin{equation*}
 \frac{ \partial ( L - E ) }{ \partial t } = [ L - E, A ] + B ( L - E ),
\end{equation*}
where $ [ \cdot, \cdot ] $ is the commutator.

Note that both Kadomtsev--Petviashvili equations can be obtained from (\ref{NV}) by considering an appropriate limit $ E \to \pm \infty $ (see \cite{ZS}, \cite{G2}).

We will consider the Cauchy problem for equation (\ref{NV}) with the initial data
\begin{equation}
\label{init_value}
 v( x, 0 ) = v_0( x ), \quad w( x, 0 ) = w_0( x ).
\end{equation}
We will assume that the function $ v_0( x ) $ satisfies the following conditions
\begin{subequations}
\label{v_init_conditions}
\begin{align}
  \label{real}
  & \bullet v_0 = \bar v_0, \\
  \label{empty_E}
  & \begin{aligned}
  \bullet \; & \mathcal{E} = \emptyset, \text{ where $ \mathcal{E} $ is the set of zeros of the Fredholm determinant $ \Delta $ of (\ref{fr_det})} \\
    & \text{for equation (\ref{fr_equation}) with $ v( x ) = v_0( x ) $},
    \end{aligned}\\
  \label{reg_decay}
  & \bullet v_0 \in \mathcal{S}( \mathbb{R}^2 ), \text{ where $ \mathcal{S} $ denotes the Schwartz class. }
\end{align}
\end{subequations}
As for the function $ w_0( x ) $, which plays an auxiliary role, we will assume that it is a continuous function decaying at infinity and determined using $ \partial_{ \bar z } w_0( x ) = -3 \partial_{ z } v_0( x ) $ from (\ref{NV}).

Condition (\ref{empty_E}) is equivalent to non--singularity of scattering data for $ v_0( x ) $.  Conditions (\ref{v_init_conditions}) define the class of initial values for which the direct and inverse scattering equations (\ref{mu_int_equation}), (\ref{mu_continuity})--(\ref{mu_limits}), with time dynamics given by (\ref{t_dynamics_b}), are everywhere solvable and the corresponding solution $ v $ of (\ref{NV}) belongs to $ C^{ \infty }( \mathbb{R}^2, \mathbb{R} ) $. We will call such solution $ ( v( x, t ), w( x, t )  )$, constructed from $ ( v_0( x ), w_0( x ) ) $ via the inverse scattering method, an ``inverse scattering solution'' of (\ref{NV}).

The main result of this paper consists in the following: we show that for the ``inverse scattering solution'' $ v( x, t ) $ of (\ref{NV}), (\ref{init_value}), where $ E < 0 $ and $ v( x, 0 ) = v_0( x ) $ satisfies (\ref{v_init_conditions}), the following estimate holds
\begin{equation}
\label{main_estimate}
 | v( x, t ) | \leqslant \frac{ \const( v_0 ) \ln( 3 + | t | ) }{ ( 1 + | t | )^{ 3/4 } }, \quad t \in \mathbb{R}, \text{ uniformly on } x \in \mathbb{R}^2.
\end{equation}
We show that this estimate is optimal in the sense that for some initial values $ v( x, 0 ) $ and for some lines $ x = \omega t $, $ \omega \in \mathbb{S}^1 $, the exact asymptotics of $ v( x, t ) $ along these lines is $ \frac{ \const }{ ( 1 + | t | )^{ 3/4 } } $ as $ | t | \to \infty $ (where the constant is nonzero). Note that de facto the ``inverse scattering solution'' is the rapidly decaying solution in the sense of (\ref{sol_definition}).

This work is a continuation of the studies on the large time asymptotic behavior of the solution of the Cauchy problem for the Novikov--Veselov equation started in \cite{KN1} for the case of positive energy $ E $. It was shown in \cite{KN1} that if the initial data $ ( v_0( x ) $, $ w_0( x ) ) $ satisfy the following conditions:
\begin{itemize}
\item $ ( v_0( x ) $, $ w_0( x ) ) $  are sufficiently regular and decaying at $ | x | \to \infty $,
\item $ v_0( x ) $ is transparent for (\ref{schrodinger}) at $ E = E_{ fixed } > 0 $, i.e. its scattering amplitude $ f $ is identically zero at fixed energy,
\item the additional ``scattering data'' $ b $ for $ v_0( x ) $ is non--singular,
\end{itemize}
then the corresponding solution of (\ref{NV}),(\ref{init_value}) can be estimated as
\begin{equation*}
 | v( x, t ) | \leqslant \frac{ \const \cdot \ln( 3 + | t | ) }{ 1 + | t | }, \quad t \in \mathbb{R} \text{ uniformly on } x \in \mathbb{R}^2.
\end{equation*}
This estimate implies, in particular, that there are no localized soliton--type traveling waves in the asymptotics of (\ref{NV}) with the``transparent'' at $ E = E_{ fixed } > 0 $ Cauchy data from the aforementioned class, in contrast with the large time asymptotics for solutions of the KdV equation with reflectionless initial data.

It was shown in \cite{N2} that all soliton--type (traveling wave) solutions of (\ref{NV}) with $ E > 0 $ must have a zero scattering amplitude at fixed energy; in addition it was proved in \cite{N2} that for the equation (\ref{NV}) with $ E > 0 $ no exponentially localized soliton--type solutions exist (even if the scattering data are allowed to have singularities). However, in \cite{G1}, \cite{G2} a family of algebraically localized solitons (traveling waves) was constructed de facto (see also \cite{KN2}). We note that for the case $ E < 0 $, though the absence of exponentially--localized solitons has been proved (see \cite{KN3}), the existence of bounded algebraically localized solitons is still an open question.

Note that studies on the large time asymptotics for solutions of the Cauchy problem for the Kadomtsev--Petviashvili equations were fulfilled in \cite{MST}, \cite{HNS}, \cite{K}.

The proofs provided in the present paper are based on the scheme developed in \cite{KN1}, the stationary phase method (see \cite{Fe}). These proofs include, in particular, an analysis of some cubic algebraic equation depending on a complex parameter.

This work was fulfilled in the framework of research carried out under the supervision of R.G. Novikov.

\section{Inverse ``scattering'' transform for the two--dimensional Schr\"o\-dinger equation at a fixed negative energy}
\label{transparent_potentials}
In this section we give a brief description of the inverse ``scattering'' transform for the two-dimensional Schr\"odinger equation (\ref{schrodinger}) at a fixed negative energy $ E $ (see \cite{GN}, \cite{N1}, \cite{G2}).

First of all, we note that by scaling transform we can reduce the scattering problem with an arbitrary fixed negative energy to the case when $ E = -1 $. Therefore, in our further reasoning we will assume that $ E = -1 $.

Let us consider potentials $ v( x ) $ for the problem (\ref{schrodinger}) satisfying the following conditions
\begin{equation}
 \label{weak_conditions}
  v = \bar v, \quad | v( x ) | \leqslant \frac{ q }{ ( 1 + | x | )^{ 2 + \varepsilon } }, \quad x \in \mathbb{R}^2,
\end{equation}
for some fixed $ q $ and $ \varepsilon > 0 $. Then it is known that for $ \lambda \in \mathbb{C} \backslash ( 0 \cup \mathcal{E} ) $, where
\begin{equation}
\label{e_set}
\begin{aligned}
& \mathcal{E} \text{ is the set of zeros of the modified Fredholm determinant } \Delta \\
& \text{ for equation (\ref{fr_equation}),}
\end{aligned}
\end{equation}
there exists a unique continuous solution $ \psi( z, \lambda ) $ of (\ref{schrodinger}) with the following asymptotics
\begin{equation}
\label{psi_mu}
\psi( z, \lambda ) = e^{ -\frac{ 1 }{ 2 }( \lambda \bar z + z / \lambda ) } \mu( z, \lambda ), \quad \mu( z, \lambda ) = 1 + o( 1 ), \quad | z | \to \infty.
\end{equation}
Here the notation $ z = x_1 + i x_2 $ is used. These solutions are known as the Faddeev solutions for the Schrodinger equation (\ref{schrodinger}), $ E = - 1 $, see for example \cite{Fa}, \cite{N1}.

The function $ \mu( z, \lambda ) $ satisfies the following integral equation
\begin{align}
\label{mu_int_equation}
& \mu( z, \lambda ) = 1 + \iint\limits_{ \zeta \in \mathbb{C} } g( z - \zeta, \lambda ) v( \zeta ) \mu( \zeta, \lambda ) d \Re \zeta d \Im \zeta \\
\label{green}
& g( z, \lambda ) = - \left( \frac{ 1 }{ 2 \pi } \right)^2 \iint\limits_{ \zeta \in \mathbb{C} } \frac{ \exp( i / 2 ( \zeta \bar z + \bar \zeta z ) ) }{ \zeta \bar \zeta + i ( \lambda \bar \zeta + \zeta / \lambda ) } d \Re \zeta d \Im \zeta,
\end{align}
where $ z \in \mathbb{C} $, $ \lambda \in \mathbb{C} \backslash 0 $.

In terms of  $ m( z, \lambda ) = ( 1 + | z | )^{ -( 2 + \varepsilon ) / 2 } \mu( z, \lambda ) $ equation (\ref{mu_int_equation}) takes the form
\begin{equation}
\label{fr_equation}
m( z, \lambda ) = ( 1 + | z | )^{ -( 2 + \varepsilon ) / 2 } + \iint\limits_{ \zeta \in \mathbb{C} } ( 1 + | z | )^{ -( 2 + \varepsilon ) / 2 } g( z - \zeta, \lambda ) \frac{ v( \zeta ) }{ ( 1 + | \zeta | )^{ -( 2 + \varepsilon ) / 2 } } m( \zeta, \lambda ) d \Re \zeta d \Im \zeta,
\end{equation}
where $ z \in \mathbb{C} $, $ \lambda \in \mathbb{C} \backslash 0 $. In addition, $ A( \cdot, \cdot, \lambda ) \in L^2( \mathbb{C} \times \mathbb{C} ) $, $ | \Tr A^2( \lambda ) | < \infty $, where $ A( z, \zeta, \lambda ) $ is the Schwartz kernel of the integral operator $ A( \lambda ) $ of the integral equation (\ref{fr_equation}). Thus, the modified Fredholm determinant for (\ref{fr_equation}) can be defined by means of the formula:
\begin{equation}
\label{fr_det}
\ln \Delta( \lambda ) = \Tr( \ln( I - A( \lambda ) ) + A( \lambda ) )
\end{equation}
(see \cite{GK} for more precise sense of such definition).

Taking the subsequent members in the asymptotic expansion (\ref{psi_mu}) for $ \psi( z, \lambda ) $, we obtain (see \cite{N1}):
\begin{multline}
\label{ab_def}
\psi( z, \lambda ) = \exp\left( - \frac{ 1 }{ 2 } \left( \lambda \bar z + \frac{ z }{ \lambda } \right) \right) \Biggl\{ 1 - 2 \pi \sgn ( 1 - \lambda \bar \lambda ) \times \\
\times \left( \frac{ i \lambda a( \lambda ) }{ z - \lambda^2 \bar z } + \exp\left( - \frac{ 1 }{ 2 } \left( \left( \frac{ 1 }{ \bar \lambda } - \lambda \right) \bar z + \left( \frac{ 1 }{ \lambda } - \bar \lambda \right) z \right) \right) \frac{ \bar \lambda b( \lambda ) }{ i ( \bar \lambda^2 z - \bar z ) } \right) + o\left( \frac{ 1 }{ | z | } \right)\Biggl\},
\end{multline}
$ | z | \to \infty $, $ \lambda \in \mathbb{C} \backslash ( \mathcal{E} \cup 0 ) $.

The functions $ a( \lambda ) $, $ b( \lambda ) $ from (\ref{ab_def}) are called the ``scattering'' data for the problem (\ref{schrodinger}), (\ref{weak_conditions}) with $ E = - 1 $.

The function $ \mu( z, \lambda ) $, defined by (\ref{mu_int_equation}), satisfies the following properties:
\begin{equation}
\label{mu_continuity}
\mu( z, \lambda ) \text{ is a continuous function of } \lambda \text{ on } \mathbb{C} \backslash ( 0 \cup \mathcal{E} );
\end{equation}

\begin{subequations}
\label{dbar}
\begin{align}
\label{mu_dbar}
& \frac{ \partial \mu( z, \lambda ) }{ \partial \bar \lambda } = r( z, \lambda ) \overline{ \mu( z, \lambda ) }, \\
\label{r}
& r( z, \lambda ) = r( \lambda ) \exp\left( \frac{ 1 }{ 2 } \left( \left( \lambda - \frac{ 1 }{ \bar \lambda } \right) \bar z - \left( \bar \lambda - \frac{ 1 }{ \lambda } \right) z \right) \right), \\
\label{t}
& r( \lambda ) = \frac{ \pi \sgn( 1 - \lambda \bar \lambda ) }{ \bar \lambda } b( \lambda )
\end{align}
\end{subequations}
for $ \lambda \in \mathbb{C} \backslash ( 0 \cup \mathcal{E} ) $;

\begin{equation}
\label{mu_limits}
\mu \to 1, \text{ as } \lambda \to \infty, \; \lambda \to 0.
\end{equation}

The function $ b( \lambda ) $ possesses the following symmetries:
\begin{equation}
\label{b_sym}
b\left( - \frac{ 1 }{ \overline{\lambda} } \right) = b( \lambda ), \quad b\left( \frac{ 1 }{ \overline{\lambda} } \right) = \overline{ b( \lambda ) }, \quad \lambda \in \mathbb{C} \backslash 0.
\end{equation}

In addition, the following theorem is valid:
\begin{theorem}[see \cite{GN}, \cite{N1}, \cite{G2}]
\label{direct_inverse}
\begin{itemize}
\rule{0pt}{0pt}
\item[(i)] Let $ v $ satisfy (\ref{v_init_conditions}). Then the scattering data $ b( \lambda ) $ for the potential $ v( z ) $ satisfy properties (\ref{b_sym}) and $ b \in \mathcal{S}( \bar D_- ) $, where $ D_- = \{ \lambda \in \mathbb{C} \colon | \lambda | > 1 \} $, $ \bar D_- = D_- \cup \partial D_- $ and $ \mathcal{S} $ denotes the Schwartz class.

\item[(ii)] Let $ b $ be a function on $ \mathbb{C} $, such that $ b \in \mathcal{S}( \bar D_- ) $ and the symmetry properties (\ref{b_sym}) hold. Then the equations of inverse scattering (\ref{mu_solve})--(\ref{v_main_formula}) are uniquely solvable and the corresponding potential $ v( z ) $ satisfies the following properties: $ v \in C^{ \infty }( \mathbb{C} ) $, $ v = \bar v $, $ | v( z ) | \to 0 $ when $ | z | \to \infty $.

\end{itemize}
\end{theorem}
Let us denote by $ T $ the unit circle on the complex plane:
\begin{equation}
\label{t_definition}
T = \{ \lambda \in \mathbb{C} \colon | \lambda | = 1 \}.
\end{equation}
Then, in addition, it is known that under the assumptions of item (i) of Theorem \ref{direct_inverse} the function $ b $ is continuous on $ \mathbb{C} $ and its derivative $ \partial_{ \lambda }b( \lambda ) $ is bounded on $ \mathbb{C} $, though discontinuous, in general, on $ T $.

Finally, if $ ( v( z, t ), w( z, t ) ) $ is a solution of equation (\ref{NV}) with $ E = -1 $ in the sense of (\ref{sol_definition}), then the dynamics of the ``scattering'' data is described by the following equations (see \cite{GN})
\begin{align}
\label{t_dynamics_a}
& a( \lambda, t ) = a( \lambda, 0 ), \\
\label{t_dynamics_b}
& b( \lambda, t ) = \exp\left\{ \left( \lambda^3 + \frac{ 1 }{ \lambda^3 } - \bar \lambda^3 - \frac{ 1 }{ \bar \lambda^3 } \right) t \right\} b( \lambda, 0 ).
\end{align}

The reconstruction of the potential $ v( z, t ) $ from these ``scattering'' data is based on the following scheme.
\begin{enumerate}
\item Function $ \mu( z, \lambda, t ) $ is constructed as the solution of the following integral equation
\begin{equation}
\label{mu_solve}
\mu( z, \lambda, t ) = 1 - \frac{ 1 }{ \pi } \iint_{ \mathbb{C} } r( z, \zeta, t ) \overline{ \mu( z, \zeta, t ) } \frac{ d \Re \zeta d \Im \zeta }{ \zeta - \lambda }.
\end{equation}

\item Expanding $ \mu( z, \lambda, t ) $ as $ \lambda \to \infty $,
\begin{equation}
\label{mu_expansion}
\mu( z, \lambda, t ) = 1 + \frac{ \mu_{ -1 }( z, t ) }{ \lambda } + o\left( \frac{ 1 }{ | \lambda | } \right),
\end{equation}
we define $ v( z, t ) $ as
\begin{equation}
\label{v_main_formula}
v( z, t ) = - 2 \partial_{ z } \mu_{ -1 }( z, t ).
\end{equation}
\item It can be shown that
\begin{equation*}
L \psi = E \psi
\end{equation*}
where
\begin{gather*}
L = - 4 \partial_{ z } \partial_{ \bar z } + v( z, t ), \quad \overline{ v( z, t ) } = v( z, t ), \quad E = - 1\\
\psi( z, \lambda, t ) = e^{ -\frac{ 1 }{ 2 } ( \lambda \bar z + z / \lambda ) } \mu( z, \lambda, t ), \quad \lambda \in \mathbb{C}, \quad z \in \mathbb{C}, \quad t \in \mathbb{R}.
\end{gather*}
\end{enumerate}

\section{Estimate for the linearized case}
\label{lin_section}
Consider
\begin{equation}
\label{lin_solution}
\begin{aligned}
& I( t, z ) = \iint\limits_{ \mathbb{C} } f( \zeta ) \exp( S( \zeta, z, t ) ) d \Re \zeta d \Im \zeta, \\
& J( t, z ) = 3 \iint\limits_{ \mathbb{C} } \frac{ \zeta }{ \bar \zeta } f( \zeta ) \exp( S( \zeta, z, t ) ) d \Re \zeta d \Im \zeta,
\end{aligned}
\end{equation}
where $ z \in \mathbb{C}$, $ t \in \mathbb{R} $,  $ f( \zeta ) \in L^1( \mathbb{C} ) $, $ S $ is defined by
\begin{equation}
\label{s_definition}
S( \lambda, z, t ) = \frac{ 1 }{ 2 }\left( \left( \lambda - \frac{ 1 }{ \bar \lambda } \right) \bar z - \left( \bar \lambda - \frac{ 1 }{ \lambda } \right) z \right) + t \left( \lambda^3 + \frac{ 1 }{ \lambda^3 } - \bar \lambda^3 - \frac{ 1 }{ \overline \lambda^3 } \right).
\end{equation}

We will also assume that $ f( \zeta ) $ satisfies the following conditions
\begin{subequations}
\label{f_assumptions}
\begin{align}
\label{f_bounded}
& f \in C^{ \infty }( \bar D_+ ), \quad f \in C^{ \infty }( \bar D_- ), \\
\label{f_decrease}
& \partial^{ m }_{ \lambda } \partial^{ n }_{ \bar \lambda } f( \lambda ) =
\begin{cases}
& O\left( | \lambda |^{ -\infty } \right) \quad \text{ as } | \lambda | \to \infty, \\
& O\left( | \lambda |^{ \infty } \right) \quad \text{ as } | \lambda | \to 0,
\end{cases} \text{ for all $ m, n \geqslant 0 $,}
\end{align}
\end{subequations}
where
\begin{equation}
\label{d_sets}
D_+ = \{ \lambda \in \mathbb{C} \colon 0 < | \lambda | \leqslant 1 \}, \quad D_- = \{ \lambda \in \mathbb{C} \colon | \lambda | \geqslant 1 \},
\end{equation}
and $ \bar D_+ = D_+ \cup T $, $ \bar D_- = D_- \cup T $ with $ T $ defined by (\ref{t_definition}).

Note that if $ v( z, t ) = I( t, z ) $, $ w( z, t ) = J( t, z ) $, where
\begin{equation*}
( | \zeta |^3 + | \zeta |^{ -3 } ) f( \zeta ) \in L^{ 1 }( \mathbb{C} ) \text{ as a function of $ \zeta $,}
\end{equation*}
and, in addition,
\begin{equation*}
\overline{ f( \zeta ) } = f( - \zeta ) \quad \text{and/or} \quad \overline{f( \zeta )} = -| \zeta |^{-4} f\left( \frac{ 1 }{ \bar \zeta } \right),
\end{equation*}
then $ v $, $ w $ satisfy the linearized Novikov--Veselov equation (\ref{NV}) with $ E = - 1 $. Besides, if
\begin{equation}
\label{main_form}
f( \zeta ) = \frac{ \pi | 1 - \zeta \bar \zeta | }{ 2 | \zeta |^2 } b( \zeta ),
\end{equation}
where $ b( \zeta ) $ is the scattering data for the initial functions $ ( v_0( z ), w_0( z ) ) $ of the Cauchy problem (\ref{NV}), (\ref{init_value}), then the integrals $ I( t, z ) $, $ J( t, z ) $ of (\ref{lin_solution}) represent the approximation of the solution $ ( v( z, t ), w( z, t ) ) $ under the assumption that $ \parallel v \parallel \ll 1 $.

The goal of this section is to give, in particular, a uniform estimate of the large--time behavior of the integral $ I( t, z ) $ of (\ref{lin_solution}).

For this purpose we introduce parameter $ u = \frac{ z }{ t } $ and write the integral $ I $ in the following form
\begin{equation}
\label{new_form}
I( t, u ) = \iint\limits_{ \mathbb{C} } f( \zeta ) \exp( t S( u, \zeta ) ) d \Re \zeta d \Im \zeta,
\end{equation}
where
\begin{equation}
\label{s_function}
S( u, \zeta ) = \frac{ 1 }{ 2 } \left( \left( \zeta - \frac{ 1 }{ \bar \zeta } \right) \bar u - \left( \bar \zeta - \frac{ 1 }{ \zeta } \right) u \right) + \left( \zeta^3 - \bar \zeta^3 + \frac{ 1 }{ \zeta^3 } - \frac{ 1 }{ \bar \zeta^3 } \right).
\end{equation}

We will start by studying the properties of the stationary points of the function $ S( u, \zeta ) $ with respect to $ \zeta $. These points satisfy the equation
\begin{equation}
\label{stationary}
S'_{ \zeta } = \frac{ \bar u }{ 2 } - \frac{ u }{ 2 \zeta^2 } + 3 \zeta^2 - \frac{ 3 }{ \zeta^4 } = 0.
\end{equation}

The degenerate stationary points obey additionally the equation
\begin{equation}
\label{degenerate}
S''_{ \zeta \zeta } = \frac{ u }{ \zeta^3 } + 6 \zeta + \frac{ 12 }{ \zeta^5 } = 0.
\end{equation}

We denote $ \xi = \zeta^2 $ and
\begin{equation*}
Q( u, \xi ) = \frac{ u }{ 2 } - \frac{ \bar u }{ 2 \xi } + 3 \xi - \frac{ 3 }{ \xi^2 }.
\end{equation*}
For each $ \xi $, a root of the function $ Q( u, \xi ) $, there are two corresponding stationary points of $ S( u, \zeta ) $, $ \zeta = \pm \sqrt{ \xi } $.

The function $ S'_{ \zeta }( u, \zeta ) $ can be represented in the following form
\begin{equation}
\label{sprime_representation}
S'_{ \zeta }( u, \zeta ) = \frac{ 3 }{ \zeta^4 } ( \zeta^2 - \zeta_0^2( u ) ) ( \zeta^2 - \zeta_1^2( u ) ) ( \zeta^2 - \zeta_2^2( u ) ).
\end{equation}

We will also use hereafter the following notations:
\begin{equation*}
\mathcal{U} = \{ u = - 6 ( 2 e^{ - i \varphi } + e^{ 2 i \varphi } ) \colon \varphi \in [ 0, 2 \pi ) \}
\end{equation*}
and
\begin{equation*}
\mathbb{U} = \{ u = r e^{ i \psi } \colon \psi = \Arg( -6( 2 e^{ - i \varphi } + e^{ 2 i \varphi } ) ), \; 0 \leqslant r \leqslant | 6 ( 2 e^{ i \varphi } + e^{ - 2 i \varphi } ) |, \; \varphi \in [ 0, 2 \pi ) \},
\end{equation*}
the domain limited by the curve $ \mathcal{U} $.

\begin{lemma}[see \cite{KN1}]
\label{lin_lemma}
\rule{1pt}{0pt}

\begin{enumerate}
\item If $ u = - 18 e^{ \frac{ 2 \pi i k }{ 3 } } $, $ k = 0, 1, 2 $, then
\begin{equation}
\label{case_one_degen}
\zeta_0( u ) = \zeta_1( u ) = \zeta_2( u ) = e^{ \frac{ \pi i k }{ 3 } }
\end{equation}
and $ S( u, \zeta ) $ has two degenerate stationary points, corresponding to a third-order root of the function $ Q( u, \xi ) $, $ \xi_1 = e^{ \frac{ 2 \pi i k }{ 3 } } $.

\item If $ u \in \mathcal{U} $ $($ i.e. $ u = -6( 2 e^{ -i \varphi } + e^{ 2 i \varphi } ) $ $)$ and $ u \neq - 18 e^{ \frac{ 2 \pi i k }{ 3 } } $, $ k = 0, 1, 2 $, then
\begin{equation}
\label{case_two_degen}
\zeta_0( u ) = \zeta_1( u ) = e^{ -i \varphi / 2 }, \quad \zeta_2( u ) = e^{ i \varphi }.
\end{equation}

Thus $ S( u, \zeta ) $ has two degenerate stationary points, corresponding to a second-order root of the function $ Q( u, \xi ) $, $ \xi_1 = e^{ -i \varphi } $, and two non--degenerate stationary points corresponding to a first-order root, $ \xi_2 = e^{ 2 i \varphi } $.

\item If $ u \in \inters \mathbb{U} = \mathbb{U} \backslash \partial \mathbb{U} $, then
\begin{equation}
\label{case_non_degen_bound}
\zeta_i( u ) = e^{ - i \varphi_i } \text{ for some real } \varphi_i, \quad \text{and} \quad \zeta_i( u ) \neq \zeta_j( u ) \quad \text{for} \quad i \neq j.
\end{equation}
In this case the stationary points of $ S( u, \zeta ) $ are non-degenerate and correspond to the roots of the function $ Q( u, \xi ) $ with absolute values equal to 1.

\item If $ u \in \mathbb{C} \backslash \mathbb{U} $, then
\begin{equation}
\label{introduce_omega}
\zeta_0( u ) = ( 1 + \omega ) e^{ -i \varphi / 2 }, \quad \zeta_1( u ) = e^{ i \varphi }, \quad  \zeta_2( u ) = ( 1 + \omega )^{ -1 } e^{ -i \varphi / 2 }
\end{equation}
for certain real values $ \varphi $ and $ \omega > 0 $.

In this case the stationary points of the function $ S( u, \zeta ) $ are non-degenerate, and correspond to the roots of the function $ Q( u, \xi ) $ that can be expressed as $ \xi_0 = ( 1 + \tau ) e^{ -i \varphi } $, $ \xi_1 = e^{ 2 i \varphi } $, $ \xi_2 = ( 1 + \tau )^{ - 1 } e^{ -i \varphi } $, $ ( 1 + \tau ) = ( 1 + \omega )^2 $.

\end{enumerate}
\end{lemma}

Formula (\ref{sprime_representation}) and Lemma \ref{lin_lemma} give a complete description of the stationary points of the function $ S( u, \zeta ) $.

In order to estimate the large--time behavior of the integral having the form
\begin{equation}
\label{more_complex_integral}
I( t, u, \lambda ) = \iint\limits_{ \mathbb{C} } f( \zeta, \lambda ) \exp( t S( u, \zeta ) ) d \Re \zeta d \Im \zeta
\end{equation}
(where $ S( u, \zeta ) $ is an imaginary--valued function) uniformly on $ u, \lambda \in \mathbb{C} $, in the present and the following sections we will use the following general scheme.
\begin{enumerate}
\item Consider $ D_{ \varepsilon } $, the union of disks with a radius of $ \varepsilon $ and centers in singular points of function $ f( \zeta, \lambda ) $ and stationary points of $ S( u, \zeta ) $ with respect to $ \zeta $.
\item Represent $ I( t, u, \lambda ) $ as the sum of integrals over $ D_{ \varepsilon } $ and $ \mathbb{C} \backslash D_{ \varepsilon } $:
\begin{equation}
\label{int_sum}
\begin{aligned}
& I( t, u, \lambda ) = I_{ int } + I_{ ext }, \quad \text{ where } \\
& I_{ int } = \iint\limits_{ D_{ \varepsilon } } f( \zeta, \lambda ) \exp( t S( u, \zeta ) ) d \Re \zeta d \Im \zeta, \\
& I_{ ext } = \iint\limits_{ \mathbb{ C } \backslash D_{ \varepsilon } } f( \zeta, \lambda ) \exp( t S( u, \zeta ) ) d \Re \zeta d \Im \zeta.
\end{aligned}
\end{equation}
\item Find an estimate of the form
\begin{equation*}
| I_{ int } | = O\left( \varepsilon^{ \alpha } \right), \quad \text{as} \quad \varepsilon \to 0 \quad ( \alpha \geqslant 1 )
\end{equation*}
uniformly on $ u $, $ \lambda $, $ t $.
\item Integrate $ I_{ ext } $ by parts using Stokes formula
\begin{multline}
\label{ext_by_parts}
I_{ ext } =  - \frac{ 1 }{ 2 i t } \int\limits_{ \partial D_{ \varepsilon } } \frac{ f( \zeta, \lambda ) \exp( t S( u, \zeta ) ) }{ S'_{ \zeta }( u, \zeta ) } d \bar \zeta - \\
 -\frac{ 1 }{ 2 i t } \int\limits_{ T \backslash D_{ \varepsilon } } \frac{ ( f_+( \zeta, \lambda ) - f_-( \zeta, \lambda ) ) \exp( t S( u, \zeta ) ) }{ S'_{ \zeta }( u, \zeta ) } d \bar \zeta
- \frac{ 1 }{ t } \iint\limits_{ \mathbb{C} \backslash D_{ \varepsilon } } \frac{ f'_{ \zeta }( \zeta, \lambda ) \exp( t S( u, \zeta ) ) }{ S'_{ \zeta }( u, \zeta ) } d \Re \zeta d \Im \zeta + \\
 +\frac{ 1 }{ t } \iint\limits_{ \mathbb{C} \backslash D_{ \varepsilon } } \frac{ f( \zeta, \lambda ) \exp( t S( u, \zeta ) ) S''_{ \zeta \zeta }( u, \zeta ) }{ ( S'_{ \zeta }( u, \zeta ) )^2 } d \Re \zeta d \Im \zeta
= - \frac{ 1 }{ t } ( I_1 + I_2 + I_3 - I_4 ),
\end{multline}
where $ f_{ \pm }( \zeta, \lambda ) = \lim\limits_{ \delta \to 0 } f( \zeta( 1 \mp \delta ), \lambda ) $, $ \zeta \in T $ and $ T $ is defined by (\ref{t_definition}).

\item For each $ I_i $ find an estimate of the form
\begin{equation*}
| I_i | = O\left( \frac{ 1 }{ \varepsilon^{ \beta } } \right), \quad \text{ as } \varepsilon \to 0.
\end{equation*}
\item Set $ \varepsilon = \dfrac{ 1 }{ ( 1 + | t | )^{ k } } $, where $ k( \alpha + \beta ) = 1 $, which yields the overall estimate
\begin{equation*}
| I( t, u, \lambda ) | = O\left( \dfrac{ 1 }{ ( 1 + | t | )^{ \frac{ \alpha }{ \alpha + \beta } } } \right), \quad \text{as} \quad | t | \to \infty.
\end{equation*}
\end{enumerate}

Using this scheme we obtain, in particular, the following result
\begin{lemma}
\label{lin_estimate_lemma}
Let a function $ f $ satisfy assumptions (\ref{f_assumptions}) and, additionally,
\begin{equation}
\label{zero_on_the_boundary}
f |_{ \,T } \equiv 0, \quad T = \{ \lambda \in \mathbb{C} \colon | \lambda | = 1 \}.
\end{equation}
Then the integral $ I( t, u ) $ of (\ref{new_form}) can be estimated
\begin{equation}
\label{lin_est}
| I( t, u ) | \leqslant \frac{ \const( f ) \ln( 3 + | t | ) }{ ( 1 + | t | )^{ 3/4 } } \quad \text{for} \quad t \in \mathbb{R}
\end{equation}
uniformly on $ u \in \mathbb{C} $.
\end{lemma}
Note that condition (\ref{zero_on_the_boundary}) is satisfied if $ f $ has the special form (\ref{main_form}). A detailed proof of Lemma \ref{lin_estimate_lemma} is given in Section \ref{proof_section}.

\section{Estimate for the non--linearized case}
\label{non_lin_section}
In this section we prove estimate (\ref{main_estimate}) for the solution $ v( x, t ) $ of the Cauchy problem for the Novikov--Veselov equation at negative energy with the initial data $ v( x, 0 ) $ satisfying properties (\ref{v_init_conditions}).

We proceed from the formulas (\ref{mu_expansion}), (\ref{v_main_formula}) for the potential $ v( z, t ) $ and the integral equation (\ref{mu_solve}) for $ \mu( z, \lambda, t ) $.

We write (\ref{mu_solve}) as
\begin{equation}
\label{mu_int}
\mu( z, \lambda, t ) = 1 + ( A_{ z, t } \mu )( z, \lambda, t ),
\end{equation}
where
\begin{equation*}
( A_{ z, t } f )( \lambda ) = \partial_{ \bar \lambda }^{ -1 }( r( \lambda ) \exp( i t S( u, \lambda ) ) \overline{ f( \lambda ) } ) = - \frac{ 1 }{ \pi } \iint\limits_{ \mathbb{C} } \frac{ r( \zeta ) \exp( t S( u, \zeta ) ) \overline{ f( \zeta ) } }{ \zeta - \lambda } d \Re \zeta d \Im \zeta
\end{equation*}
and $ S( u, \zeta ) $ is defined by (\ref{s_function}), $ u = \dfrac{ z }{ t } $.

Equation (\ref{mu_int}) can be also written in the form
\begin{equation}
\label{mu_an_form}
\mu( z, \lambda, t ) = 1 + A_{ z, t } \cdot 1 + ( A^2_{ z, t } \mu )( z, \lambda, t ).
\end{equation}
According to the theory of the generalized analytic functions (see \cite{V}), equations (\ref{mu_int}), (\ref{mu_an_form}) have a unique bounded solution for all $ z, t $. This solution can be written as
\begin{equation}
\label{mu_sol}
\mu( z, \lambda, t ) = ( I - A_{ z, t }^2 )^{ -1 }( 1 + A_{ z, t } \cdot 1 ).
\end{equation}
Equation (\ref{mu_sol}) implies the following formal asymptotic expansion
\begin{equation}
\label{expanded}
\mu( z, \lambda, t ) = ( I + A_{ z, t }^2 + A_{ z, t }^4 + \ldots )( 1 + A_{ z, t } \cdot 1 ).
\end{equation}

We also introduce functions $ \nu( z, \lambda, t ) = \partial_{ z } \mu( z, \lambda, t ) $ and $ \eta( z, \lambda, t ) = \partial_{ \bar z } \mu( z, \lambda, t ) $. In terms of these functions the potential $ v( z, t ) $ is obtained by the formula
\begin{equation}
\label{v_by_nu}
v( z, t ) = - 2 \nu_{ -1 }( z, t ),
\end{equation}
where $ \nu_{ -1 }( z, t ) $ is defined by expanding $ \nu( z, \lambda, t ) $ as $ | \lambda | \to \infty $:
\begin{equation*}
\nu( z, \lambda, t ) = \frac{ \nu_{ -1 }( z, t ) }{ \lambda } + o\left( \frac{ 1 }{ | \lambda | } \right), \quad | \lambda | \to \infty.
\end{equation*}

The pair of function $ \nu( z, \lambda, t ) $, $ \eta( z, \lambda, t ) $ satisfy the following system of differential equations:
\begin{equation}
\label{nu_eta_dif}
\begin{cases}
& \dfrac{ \partial \nu( z, \lambda, t ) }{ \partial \bar \lambda } = \partial_{ z } r( z, \lambda, t ) \overline{ \mu( z, \lambda, t ) } + r( z, \lambda, t ) \overline{ \eta( z, \lambda, t ) }, \\
& \dfrac{ \partial \eta( z, \lambda, t ) }{ \partial \bar \lambda } = \partial_{ z } r( z, \lambda, t ) \overline{ \mu( z, \lambda, t ) } + r( z, \lambda, t ) \overline{ \nu( z, \lambda, t ) }.
\end{cases}
\end{equation}
Equations (\ref{nu_eta_dif}) can also be written in the integral form
\begin{equation}
\label{nu_eta_int}
\begin{cases}
& \nu( z, \lambda, t ) = ( B_{ z, t } \mu )( z, \lambda, t ) + ( A_{ z, t } \eta )( z, \lambda, t ), \\
& \eta( z, \lambda, t ) = ( B_{ z, t } \mu )( z, \lambda, t ) + ( A_{ z, t } \nu )( z, \lambda, t ),
\end{cases}
\end{equation}
where operator $ B_{ z, t } $ is defined
\begin{equation}
( B_{ z, t } f )( \lambda ) = \partial_{ \bar \lambda }^{ -1 }( \partial_{ z } r( z, \lambda, t ) \overline{ f( \lambda ) } ) = - \frac{ 1 }{ \pi } \iint\limits_{ \mathbb{C} } \frac{ \partial_{ z } r( z, \zeta, t ) \overline{ f( \zeta ) } }{ \zeta - \lambda } d \Re \zeta d \Im \zeta.
\end{equation}

Thus for the function $ \nu( z, \lambda, t ) $ we obtain equation
\begin{equation*}
\nu = ( B_{ z, t } + A_{ z, t } B_{ z, t } ) \mu + A_{ z, t }^2 \nu,
\end{equation*}
or the following formal asymptotic expansion
\begin{equation}
\label{nu_representation}
\nu = ( I + A_{ z, t }^2 + A_{ z, t }^4 + \ldots )( ( B_{ z, t } + A_{ z, t } B_{ z, t } ) ( I + A_{ z, t }^2 + A_{ z, t }^4 + \ldots ) ( 1 + A_{ z, t } \cdot 1 ) ).
\end{equation}
We will write this formula in the form $ \nu = B_{ z, t } \cdot 1 + A_{ z, t } B_{ z, t } \cdot 1 + R_{ z, t }( \lambda ) $.

\begin{lemma}
\label{estimates_lemma}
Let $ f( \lambda, z, t ) $ be an arbitrary testing function such that
\begin{equation*}
\label{test_function_condition}
| f | \leqslant \frac{ c_f }{ ( 1 + | t | )^{ \delta } }, \quad | \partial_{ \lambda } f | \leqslant \frac{ c_f }{ ( 1 + | t | )^{ \delta } } \quad \forall \lambda \in \mathbb{C}, z \in \mathbb{C}, t \in \mathbb{R}
\end{equation*}
with some positive constant $ c_f $ independent of $ \lambda $, $ z $, $ t $ and some $ \delta \geqslant 0 $. Then:
\begin{enumerate}
\item The following estimates hold for $ B_{ z, t } \cdot f $:
\begin{equation*}
( B_{ z, t } \cdot f )( \lambda ) = \frac{ \beta_1( z, t ) }{ \lambda } + o\left( \frac{ 1 }{ | \lambda | } \right) \text{ for } \lambda \to \infty,
\end{equation*}
where
\begin{equation*}
\beta_1( z, t ) = \frac{ 1 }{ \pi } \iint_{ \mathbb{C} } \partial_{ z } r( z, \zeta, t ) \overline{ f( \zeta, z, t ) } d \Re \zeta d \Im \zeta,
\end{equation*}
and
\begin{equation}
\label{beta1_estimate}
| \beta_1( z, t ) | \leqslant \frac{ \hat \beta_1 c_f \ln( 3 + | t | ) }{ ( 1 + | t | )^{ 3/4 + \delta } };
\end{equation}
in addition,
\begin{equation}
\label{beta2_estimate}
| ( B_{ z, t } \cdot f )( \lambda ) | \leqslant \frac{ \hat \beta_2 c_f \ln( 3 + | t | ) }{ ( 1 + | t | )^{ 1/2 + \delta } } \text{ for } \lambda \in T,
\end{equation}
where $ T $ is defined by (\ref{t_definition}), and
\begin{equation}
\label{beta3_estimate}
| ( B_{ z, t } \cdot f )( \lambda ) | \leqslant \frac{ \hat \beta_3 c_f }{ ( 1 + | t | )^{ 1/4 + \delta } } \quad \forall \lambda \in \mathbb{C}.
\end{equation}

\item The following estimates hold for $ A_{ z, t } \cdot B_{ z, t } \cdot f $:
\begin{equation*}
( A_{ z, t } \cdot B_{ z, t } \cdot f )( \lambda ) = \frac{ \alpha_1( z, t ) }{ \lambda } + o\left( \frac{ 1 }{ | \lambda | } \right) \text{ for } \lambda \to \infty,
\end{equation*}
where
\begin{equation}
\label{alpha1_integral}
\alpha_1( z, t ) = - \frac{ 1 }{ \pi^2 } \iint\limits_{ \mathbb{C} } d \Re \zeta d \Im \zeta \, r( z, \zeta, t ) \iint\limits_{ \mathbb{C} } \frac{ \partial_{ z } r( z, \eta, t ) }{ \eta - \zeta }  \overline{ f( \eta, z, t ) } d \Re \eta \, d \Im \eta,
\end{equation}
and
\begin{equation}
\label{alpha1_estimate}
| \alpha_1( z, t ) | \leqslant \frac{ \hat \alpha_1 c_f }{ ( 1 + | t | )^{ 3/4 + \delta } };
\end{equation}
in addition,
\begin{equation}
\label{alpha2_estimate}
| ( A_{ z, t } \cdot B_{ z, t } \cdot f )( \lambda ) | \leqslant \frac{ \hat \alpha_2 c_f }{ ( 1 + | t | )^{ 1/2 + \delta } } \quad \forall \lambda \in \mathbb{C}.
\end{equation}

\item The following estimates for $ A^{ n }_{ z, t } \cdot f $ hold:
\begin{equation*}
( A^{ n }_{ z, t } \cdot f )( \lambda ) = \frac{ \gamma_n( z, t ) }{ \lambda } + o\left( \frac{ 1 }{ | \lambda | } \right) \text{ for } \lambda \to \infty,
\end{equation*}
where
\begin{equation*}
\gamma_n( z, t ) = \frac{ 1 }{ \pi } \iint\limits_{ \mathbb{C} } r( z, \zeta, t ) \overline{ ( A^{ n - 1 }_{ z, t } \cdot f )( \zeta ) } d \Re \zeta d \Im \zeta
\end{equation*}
and
\begin{equation}
\label{gamma1_estimate}
| \gamma_n( z, t ) | \leqslant \frac{ ( \hat \gamma_1 )^n c_f }{ ( 1 + | t | )^{ \delta + \frac{ 1 }{ 5 } \lceil \frac{ n - 1 }{ 2 } \rceil + \frac{ 2 }{ 5 } } },
\end{equation}
where $ \lceil s \rceil $ denotes the smallest integer following $ s $. In addition,
\begin{equation}
\label{gamma2_estimate}
| ( A^{ n }_{ z, t } \cdot f )( \lambda ) | \leqslant \frac{ ( \hat \gamma_2 )^n c_f }{ ( 1 + | t | )^{ \delta + \frac{ 1 }{ 5 } \lceil \frac{ n }{ 2 } \rceil } } \quad \forall \lambda \in \mathbb{C}.
\end{equation}

\item The following estimate holds for $ R_{ z, t } $:
\begin{equation}
R_{ z, t }( \lambda ) = \frac{ q( z, t ) }{ \lambda } + o\left( \frac{ 1 }{ | \lambda | } \right) \text{ for } \lambda \to \infty,
\end{equation}
and
\begin{equation}
\label{q_estimate}
| q( z, t ) | \leqslant \frac{ \hat q( c_f ) }{ ( 1 + | t | )^{ 9/10 } }.
\end{equation}

\end{enumerate}
\end{lemma}

A detailed proof of Lemma \ref{estimates_lemma} is given in Section \ref{proof_section}.

From formulas (\ref{v_by_nu}), (\ref{nu_representation}) and Lemma \ref{estimates_lemma} follows immediately the following theorem.
\begin{theorem}
Let $ v( x, t ) $ be the ``inverse scattering solution'' of the Cauchy problem for the Novikov--Veselov equation (\ref{NV}) with $ E = -1 $ and the initial data $ v( x, 0 ) = v_0( x ) $ satisfying (\ref{v_init_conditions}). Then
\begin{equation*}
| \, v( x, t ) | \leqslant \frac{ \const( v_0 ) \ln( 3 + | t | ) }{ ( 1 + | t | )^{ 3/4 } }, \quad x \in \mathbb{R}^2, t \in \mathbb{R}.
\end{equation*}
\end{theorem}

\section{Optimality of estimates (\ref{main_estimate}) and (\ref{lin_est})}
\label{optim_section}
In this section we show that estimates (\ref{main_estimate}) and (\ref{lin_est}) are optimal in the following sense: there exists such a line $ z = \hat u t $ that along this line $ I( z, t ) $ from (\ref{lin_est}) behaves asymptotically as $ \frac{ \const }{ ( 1 + | t | )^{ 3/4 } } $ with come nonzero constant; there exist such initial data satisfying (\ref{v_init_conditions}) that the corresponding solution $ v( z, t ) $ behaves asymptotically as $ \frac{ \const }{ ( 1 + | t | )^{ 3/4 } } $ as $ | t | \to \infty $, where the constant is nonzero.

\subsection{Optimality of the estimate for the linearized case}
\label{optim_lin_section}
Let us consider the integral
\begin{equation}
\label{i_particular_form}
I( t, u ) = \iint\limits_{ \mathbb{C} } f( \zeta ) \exp( t S( u, \zeta ) ) d \Re \zeta d \Im \zeta,
\end{equation}
where
\begin{equation}
\label{f_particular_form}
f( \zeta ) =\frac{ \pi | 1 - \zeta \bar \zeta | }{ 2 | \zeta |^2 } b( \zeta ),
\end{equation}
with $ b( \zeta ) $ satisfying (\ref{f_assumptions}), and $ S( u, \zeta ) $ is defined by (\ref{s_function}), for $ u = \hat u = -18 $. As shown in Lemma \ref{lin_lemma} for this value of parameter $ u $ the phase $ S( \hat u, \zeta ) = S( \zeta ) $ has two degenerate stationary points $ \zeta = \pm 1 $ of the third order.

To calculate the exact asymptotic behavior of $ I( t, \hat u ) $ we will use the classic stationary method as described in \cite{Fe}. First of all, we note that $ f( \zeta ) $ is continuous, but not continuously differentiable on $ \mathbb{C} $. Thus we will consider separately the integrals
\begin{equation*}
I_{ + }( t ) = \iint\limits_{ D_{ + } } f( \zeta ) \exp( t S( \zeta ) ) d \Re \zeta d \Im \zeta, \quad I_{ - }( t ) = \iint\limits_{ D_{ - } } f( \zeta ) \exp( t S( \zeta ) ) d \Re \zeta d \Im \zeta,
\end{equation*}
where $ D_+ $ and $ D_- $ are defined in (\ref{d_sets}).

Let us introduce the partition of unity $ \psi_{ 1 }( \zeta ) + \psi_{ 0 }( \zeta ) + \psi_{ -1 }( \zeta ) \equiv 1 $, such that $ 0 \leqslant \psi_i \leqslant 1 $, $ \psi_{ i } \in C^{ \infty }( \mathbb{C} ) $, $ \psi_{ \pm 1 }( \zeta ) \equiv 1 $ in some neighborhood of $ \zeta = \pm 1 $, respectively, and $ \psi_{ \pm 1 }( \zeta ) \equiv 0 $ everywhere outside some neighborhood of $ \zeta = \pm 1 $ respectively. Then it is known (see \cite{Fe}) that
\begin{multline*}
I_{ + }( t ) = \iint\limits_{ D_{ + } } f( \zeta ) \psi_{ 1 }( \zeta ) \exp( t S( \zeta ) ) d \Re \zeta d \Im \zeta + \iint\limits_{ D_{ + } } f( \zeta ) \psi_{ -1 }( \zeta ) \exp( t S( \zeta ) ) d \Re \zeta d \Im \zeta + O\left( \frac{ 1 }{ | t | } \right) = \\
= I_{ + }^{ + }( t ) + I_{ + }^{ - }( t ) + O\left( \frac{ 1 }{ | t | } \right) \text{ as } | t | \to \infty.
\end{multline*}

First we will estimate $ I_{ + }^{ + }( t ) $. We note that for the phase $ S( \zeta ) $ the following representation is valid
\begin{equation*}
S( \zeta ) = P( \zeta ) - P( \bar \zeta ),
\end{equation*}
where $ P( \zeta ) $ is a holomorphic function defined by
\begin{equation*}
P( \zeta ) = -9 \zeta - \frac{ 9 }{ \zeta } + \zeta^3 + \frac{ 1 }{ \zeta^3 } + 16;
\end{equation*}
in addition, $ P_{ \zeta }( \zeta ) = S_{ \zeta }( \zeta ) = \frac{ 3 }{ \zeta^4 }( \zeta - 1 )^3( \zeta + 1 )^3 $.

Note that $ P( \zeta ) $ can be written in the form $ P( \zeta ) = \rho( \zeta ) ( \zeta - 1 )^4 $, where $ \rho( \zeta ) = \frac{ \zeta^2 + 4 \zeta + 1 }{ \zeta^3 } $ and $ \lim\limits_{ \zeta \to 1 } \rho( \zeta ) \neq 0 $. For the function $ \rho( \zeta ) $ the expression $ ( \rho( \zeta ) )^{ 1/4 } $ can be uniquely defined in some neighborhood of $ \zeta = 1 $. Further, we define the transformation $ \zeta \to \eta $:
\begin{equation*}
\eta = ( \rho( \zeta ) )^{ 1/4 } ( \zeta - 1 ).
\end{equation*}
Since we have that
\begin{equation}
\label{der_in_zero}
\left. \frac{ \partial \eta }{ \partial \zeta } \right|_{ \zeta = 1 } = \sqrt[4]{6} \neq 0,
\end{equation}
the inverse transformation $ \zeta = \varphi( \eta ) $ is defined in some small neighborhood of $ \eta = 0 $. In terms of the new variable $ \eta $ the phase can be represented
\begin{equation*}
S( \zeta ) = \eta^4 - \bar \eta^4.
\end{equation*}

Now if we denote $ x = \Re \eta $, $ y = \Im \eta $, the integral $ I_{ + }^{ + }( t ) $ becomes
\begin{equation*}
I_{ + }^{ + }( t ) = \iint\limits_{ \Delta_+ } \tilde f( x + i y ) \tilde \psi_1( x + i y ) \exp( 3 i t x y ( x^2 - y^2 ) ) | \partial_{ \eta } \varphi( x + i y ) |^2 dx dy,
\end{equation*}
where $ \tilde f = f \circ \varphi $, $ \tilde \psi_1 = \psi_{ 1 } \circ \varphi $ and $ \Delta_+ =\{ ( x, y ) \in \mathbb{R}^2 \colon x < 0 \} $, i.e. $ \Delta_+ $ is the half-plane containing the image of $ D_+ \cap B_{ \varepsilon }( 1 ) $ under the transformation $ x + i y = \varphi^{ -1 }( \zeta ) $.

The integral $ I_{ + }^{ + }( t ) $ can be written in the form
\begin{equation*}
I_{ + }^{ + }( t ) = \int\limits_{ -\infty }^{ +\infty } dc \exp( 3 i t c ) \hspace{-0.2cm} \int\limits_{ \gamma_c \cap \Delta_+ } \hspace{-0.2cm} \tilde f( x + i y ) \tilde \psi_1( x + i y ) | \partial_{ \eta } \varphi( x + i y ) |^2 d \omega_S,
\end{equation*}
where $ d \omega_S $ is the Gelfand--Leray differential form, defined as
\begin{equation}
\label{g_l_form}
d S \wedge d \omega_S = dx \wedge dy
\end{equation}
and in the particular case under study equal to
\begin{equation*}
d \omega_S = \frac{ -( x^3 - 3 x y^2 ) dx + ( 3 x^2 y - y^3 ) dy }{ ( x^2 + y^2 )^3 };
\end{equation*}
$ \gamma_c $ is an oriented contour consisting of points of the set $ \{ S( x, y ) = c \} $ with the orientation chosen so that (\ref{g_l_form}) holds.

As $ \tilde \psi_1( x + i y ) $ is equal to zero outside some $ B_R( 0 ) $, a disk of radius $ R $ centered in the origin, then there exists such $ c_* > 0 $ that the set $ \{ S( x, y ) = c \} $ lies outside $ B_R( 0 ) $ for any $ c < - c_* $, $ c > c_* $. Thus the integral $ I_{ + }^{ + }( t ) $ can be written
\begin{equation*}
I_{ + }^{ + }( t ) = \int\limits_{ -c_* }^{ c_* } dc \exp( 3 i t c ) \hspace{-0.2cm} \int\limits_{ \gamma_c \cap \Delta_+ } \hspace{-0.2cm} \tilde f( x + i y ) \tilde \psi_1( x + i y ) | \partial_{ \eta } \varphi( x + i y ) |^2 d \omega_S.
\end{equation*}

Performing the change of variables
\begin{equation*}
\begin{cases}
& x \mapsto c^{ 1/4 } x, \quad y \mapsto c^{ 1/4 } y \quad \text{ for } c > 0, \\
& x \mapsto (-c)^{ 1/4 } x, \quad y \mapsto (-c)^{ 1/4 } y \quad \text{ for } c < 0
\end{cases}
\end{equation*}
yields
\begin{equation*}
I_{ + }^{ + }( t ) = \int\limits_{ 0 }^{ c_* } \frac{ dc \exp( 3 i t c ) }{ c^{ 1/2 } } F_{ + }( c ) + \int\limits_{ -c_* }^{ 0 } \frac{ dc \exp(3  i t c ) }{ ( -c )^{ 1/2 } } F_{ - }( c ),
\end{equation*}
where
\begin{equation}
\label{F_integrals}
F_{ + }( c ) = \hspace{-0.9cm} \int\limits_{ \gamma_+ \cap \{ c^{ 1/4 }( x, y ) \in \Delta_{+} \} } \hspace{-0.9cm} \mathcal{F}( c, x, y ) d \omega_S, \quad F_{ - }( c ) = \hspace{-1.1cm} \int\limits_{ \gamma_- \cap \{ (-c)^{1/4}( x, y ) \in \Delta_+ \} } \hspace{-1.1cm} \mathcal{F}( -c, x, y ) d \omega_S,
\end{equation}
$ \gamma_+ $, $ \gamma_- $ are oriented sets consisting of points of the sets $ \{ S( x, y ) = 1 \} $, $ \{ S( x, y ) = -1 \} $ correspondingly with orientation chosen so that (\ref{g_l_form}) holds and
\begin{equation*}
\mathcal{F}( c, x, y ) = \tilde f( c^{1/4}( x + i y ) ) \tilde \psi_1( c^{1/4}( x + i y ) ) | \partial_{ \eta } \varphi( c^{1/4}( x + i y ) ) |^2.
\end{equation*}
For any fixed positive $ c $ the integrals in (\ref{F_integrals}) converge because the set $ \{ S( x, y ) = 1 \} $ is separated from zero and, consequently, the denominator does not vanish, and because $ \tilde \psi_1 $ is a function with a bounded support and thus the domains of integration in (\ref{F_integrals}) are, in fact, bounded. Besides, since $ \Delta_+ $ is a conic set, the functions $ F_+( c ) $ and $ F_-( c ) $ can be expressed as follows
\begin{equation}
\label{stop_formula}
F_{ + }( c ) = \hspace{-0.3cm} \int\limits_{ \gamma_+ \cap \Delta_{+} } \hspace{-0.3cm} \mathcal{F}( c, x, y ) d \omega_S, \quad F_{ - }( c ) = \hspace{-0.4cm} \int\limits_{ \gamma_- \cap \Delta_+ } \hspace{-0.4cm} \mathcal{F}( -c, x, y ) d \omega_S,
\end{equation}

In some neighborhood $ \mathcal{U}_0 $ of $ c^{ 1/4 }( x + i y ) = 0 $ containing the support of $ \tilde \psi_1 $ the function $ \tilde f( c^{ 1/4 }( x + i y ) ) $ can be represented as
\begin{equation*}
\tilde f( c^{ 1/4 }( x + i y ) ) = f( 1 ) + 6^{ -1/4 } [ \partial_{ \zeta } f_{ D_+ }( 1 )( x + i y ) + \partial_{ \bar \zeta } f_{ D_+ }( 1 ) ( x - i y ) ] c^{ 1/4 } + g( c^{ 1/4 }( x + i y ) ),
\end{equation*}
where $ \partial_{ \zeta } f_{ D_+ }( 1 ) = \lim\limits_{ \substack{ \zeta \in D_+ \\ \zeta \to 1 } } \partial_{ \zeta } f $, $ \partial_{ \bar \zeta } f_{ D_+ }( 1 ) = \lim\limits_{ \substack{ \zeta \in D_+ \\ \zeta \to 1 } } \partial_{ \bar \zeta } f $ and, in addition, $ g $ is a function that can be estimated
\begin{equation}
\label{g_est}
| g( c^{ 1/4 }( x + i y ) ) | \leqslant K | c^{ 1/4 }( x + i y ) |^{ 1 + \alpha }, \quad c^{ 1/4 }( x + i y ) \in \mathcal{U}_0
\end{equation}
with some constants $ \alpha > 0 $, $ K > 0 $. Using (\ref{f_particular_form}), we note that $ f( 1 ) = 0 $, $ \partial_{ \zeta } f_{ D_+ }( 1 ) = \partial_{ \bar \zeta } f_{ D_+ }( 1 ) \stackrel{ \definition }{ = } f'_{ D_+ }( 1 ) $ and thus
\begin{equation*}
\tilde f( c^{ 1/4 }( x + i y ) ) = 6^{ - 1/4 } f'_{ D_+ }( 1 ) x c^{ 1/4 } + g( c^{1/4}( x + i y ) ), \quad c^{ 1/4 }( x + i y ) \in \mathcal{U}_0.
\end{equation*}
Taking into account (\ref{der_in_zero}) we obtain
\begin{equation}
\mathcal{F}( c, x, y ) = \gamma f'_{ D_+ }( 1 ) x c^{ 1/4 } + \tilde g( c^{1/4}( x + i y ) ), \quad c^{ 1/4 }( x + i y ) \in \mathcal{U}_0,
\end{equation}
where $ \gamma = 6^{ - 3/4 } $ and $ \tilde g $ is a function satisfying an estimate similar to (\ref{g_est}).

It follows then that the functions $ F_{ \pm }( c ) $ behave asymptotically as
\begin{equation*}
F_{ \pm }( c ) = \gamma f'_{ D_+ }( 1 ) J_{ \Delta_+ }^{ \pm } ( \pm c )^{ 1/4 } + R( c ), \text{ when } c \to 0,
\end{equation*}
where
\begin{equation*}
J_{ \Delta_+ }^+ = \hspace{-0.2cm} \int\limits_{ \gamma_+ \cap \Delta_+ } \hspace{-0.2cm} x d \omega_S, \quad J_{ \Delta_+ }^- = \hspace{-0.2cm} \int\limits_{ \gamma_- \cap \Delta_+ } \hspace{-0.2cm} x d \omega_S
\end{equation*}
and $ R( c ) $ denotes the remainder.

The integrals $ J_{ \Delta_+ }^{ \pm } $ converge because the set $ \{ S( x, y ) = \pm 1 \} $ represents a combination of curves which do not pass through zero and converge either to the lines $ | y | = | x | $ or to the coordinate axes with velocities $ | y | = \frac{ 1 }{ | x |^3 } $ and $ | x | = \frac{ 1 }{ | y |^3 } $ correspondingly.

The remainder $ R( c ) $ behaves asymptotically as $ o\left( c^{ 1/4 } \right) $ because we can estimate
\begin{equation*}
| R( c ) | \leqslant \tilde K c^{ 1/4 ( 1 + \alpha ) } \hspace{ -0.5cm } \int\limits_{ ( \gamma_+ \cup \gamma_- )  \cap \Delta_+ } \hspace{ -0.5cm } | x + i y |^{ 1 + \alpha } | d \omega_S |.
\end{equation*}
and the integral converges due to the properties of the set $ \{ S( x, y )  = \pm 1 \} $ explained above.

Thus $ I_{ + }^{ + }( t ) $ behaves asymptotically as (see \cite[Chapter III, \S 1]{Fe})
\begin{equation*}
I_{ + }^{ + }( t ) = \frac{ \gamma }{ 3^{ 3/4 } } f'_{ D_+ }( 1 ) \Gamma\left( \frac{ 3 }{ 4 } \right) \left[ J_{ \Delta_+ }^{ + } \exp\left( \frac{ i \pi 3 }{ 8 } \right) + J_{ \Delta_+ }^{ - } \exp\left( -\frac{ i \pi 3 }{ 8 } \right) \right] \frac{ 1 }{ t^{ 3/4 } } + o\left( \frac{ 1 }{ t^{ 3/4 } } \right),
\end{equation*}
where $ \Gamma $ is the Gamma function.

Let us perform the same procedure for $ I_+^-( t ) $. In this case we will define $ P( \zeta ) $ as $ P( \zeta ) = - 9 \zeta - \frac{ 9 }{ \zeta } + \zeta^3 + \frac{ 1 }{ \zeta^3 } - 16 $ and obtain
\begin{equation*}
P( \zeta ) = \rho( \zeta ) ( \zeta + 1 )^4, \quad \rho( \zeta ) = \frac{ \zeta^2 - 4 \zeta + 1 }{ \zeta^3 }.
\end{equation*}
Since $ \rho( -1 ) = -6 $, we will define the following transformation $ \zeta \to \eta $ in the neighborhood of $ \zeta = -1 $: $ \eta = ( -\rho( \zeta ) )^{ 1/4 } ( \zeta + 1 ) $ (note that $ \left. \frac{ \partial \eta }{ \partial \zeta } \right|_{ \zeta = -1 } = \sqrt[ 4 ]{ 6 } $). Then $ S( \zeta ) $ can be represented $ S( \zeta ) = - \eta^4 + \bar \eta^4 $ and integral $ I_+^-( t ) $ becomes
\begin{equation*}
I_+^-( t ) = \iint\limits_{ \Delta_- } \tilde f( x + i y ) \tilde \psi_{ - 1 }( x + i y ) \exp( - 3 i t x y ( x^2 - y^2 ) ) | \partial_{ \eta } \varphi( x + i y ) |^2 d x d y,
\end{equation*}
where $ \Delta_- = \{ ( x, y ) \in \mathbb{R}^2 \colon x > 0 \} $ and the functions $ \tilde f $, $ \tilde \psi_{ -1 } $, $ \varphi $ are defined similarly to the case of $ I_+^+( t ) $. The integral $ I_+^-( t ) $ can also be written
\begin{equation*}
I_+^-( t ) = \int\limits_{ -\infty }^{ +\infty } dc \exp( - 3 i t c ) \hspace{-0.2cm} \int\limits_{ \gamma_c \cap \Delta_- } \hspace{-0.2cm} \tilde f( x + i y ) \tilde \psi_1( x + i y ) | \partial_{ \eta } \varphi( x + i y ) |^2 d \omega_S,
\end{equation*}
where $ \gamma_c $ and $ d \omega_S $ are the same as for the case of $ I_+^+( t ) $. Performing further the same procedure as for the case of $ I_+^+( t ) $ and taking into account that $ J_{ \Delta_+ }^{ + } = - J_{ \Delta_- }^{ + } $, $ J_{ \Delta_+ }^{ - } = - J_{ \Delta_- }^{ - } $, we obtain the following asymptotic expansion for $ I_+^-( t ) $:
\begin{equation*}
I_{ + }^{ - }( t ) = - \frac{ \gamma }{ 3^{ 3/4 } } f'_{ D_+ }( -1 ) \Gamma\left( \frac{ 3 }{ 4 } \right) \left[ J_{ \Delta_+ }^{ + } \exp\left( -\frac{ i \pi 3 }{ 8 } \right) + J_{ \Delta_+ }^{ - } \exp\left( \frac{ i \pi 3 }{ 8 } \right) \right] \frac{ 1 }{ t^{ 3/4 } } + o\left( \frac{ 1 }{ t^{ 3/4 } } \right).
\end{equation*}

Considering the case of $ I_-( t ) $ we note that in order to get an asymptotic representation for $ I_-^+( t ) $ and $ I_-^-( t ) $ we need to replace $ D_+ \to D_- $, $ \Delta_+ \to \Delta_- $, $ \Delta_- \to \Delta_+ $ in the formulas for $ I_+^+( t ) $ and $ I_+^-( t ) $ correspondingly. Taking into account that $ f'_{ D_+ }( 1 ) = - f'_{ D_- }( 1 ) $, $ f'_{ D_+ }( -1 ) = - f'_{ D_- }( -1 ) $, we obtain
\begin{equation*}
I( t ) = \frac{ C }{ t^{ 3/4 } } + o\left( \frac{ 1 }{ t^{ 3/4 } } \right),
\end{equation*}
where
\begin{multline}
\label{c_formula}
C = \frac{ 2 \gamma }{ 3^{ 3/4 } } \Gamma\left( \frac{ 3 }{ 4 } \right) \biggl( f'_{ D_+ }( 1 ) \left( J_{ \Delta_+ }^{ + } \exp\left( \frac{ i \pi 3 }{ 8 } \right) + J_{ \Delta_+ }^{ - } \exp\left( -\frac{ i \pi 3 }{ 8 } \right) \right) - \\
f'_{ D_+ }( -1 ) \left( J_{ \Delta_+ }^{ + } \exp\left( -\frac{ i \pi 3 }{ 8 } \right) + J_{ \Delta_+ }^{ - } \exp\left( \frac{ i \pi 3 }{ 8 } \right) \right) \biggr).
\end{multline}

Thus we have shown that the linear approximation of the solution $ v( z, t ) $ of (\ref{NV}), when $ z = -18 t $, behaves asymptotically as $ \frac{ C }{ t^{ 3/4 } } $ when $ t \to \infty $.

Note that on the set $ \gamma_+ \cup \gamma_- $ the differential form $ d \omega_S $ is positive. Thus $ J_{ \Delta_+ }^{ \pm } $ are some negative constants, and expressions $ J_{ \Delta_+ }^{ + } \exp\left( \frac{ i \pi 3 }{ 8 } \right) + J_{ \Delta_+ }^{ - } \exp\left( -\frac{ i \pi 3 }{ 8 } \right) $, $ J_{ \Delta_+ }^{ + } \exp\left( -\frac{ i \pi 3 }{ 8 } \right) + J_{ \Delta_+ }^{ - } \exp\left( \frac{ i \pi 3 }{ 8 } \right) $ do not vanish. On the other hand, from (\ref{b_sym}) and (\ref{f_particular_form}) it follows that $ f'_{ D_+ }( 1 ) / f'_{ D_+ }( -1 ) = b( 1 ) / \overline{ b( 1 ) } $. Thus, in the general case the constant $ C $ from (\ref{c_formula}) is nonzero. We have proved the optimality of the estimate (\ref{main_estimate}) in the linear approximation.

\subsection{Optimality of the estimate for the non-linear case}
Now we show that for certain initial values $ v( z, 0 ) $ the corresponding solution $ v( z, t ) $ of (\ref{NV}) behaves asymptotically as $ \frac{ c }{ t^{ 3/4 } } $ along the line $ z = -18 t $ for some $ c \neq 0 $. Let us show that the integral (\ref{alpha1_integral}) with $ f \equiv 1 $ and $ z = -18 t $ behaves as $ \frac{ \const }{ t^{ 3/4 } } $.

We can represent $ \alpha_1( -18 t, t ) $ in the form (\ref{i_particular_form}) with $ f( \zeta, t ) = r( \zeta, t ) \rho( \zeta, t ) $, where
\begin{equation*}
\rho( \zeta, t ) = - \frac{ 1 }{ \pi^2 } \iint\limits_{ \mathbb{C} } \frac{ \partial_{ z } r( \eta, z, t ) |_{ z = -18t } }{ \eta - \zeta } d \Re \eta d \Re \zeta.
\end{equation*}
In other terms,
\begin{equation*}
f( \zeta ) = f( \zeta, t ) = \frac{ \sgn( 1 - \zeta \bar \zeta ) }{ \bar \zeta } b( \zeta ) \partial_{ \bar \zeta }^{ -1 } \left[ \frac{ \pi | 1 - \zeta \bar \zeta | }{ 2 | \zeta |^2 } b( \zeta ) \exp( t S( \zeta ) ) \right].
\end{equation*}

We proceed following the scheme of estimate for $ I( t ) $ until formula (\ref{g_est}). Then we represent
\begin{equation*}
\tilde \psi_1( c^{ 1/4 }( x + i y ) ) | \partial_{ \eta } \varphi( c^{ 1/4 }( x + i y ) ) |^2 = \frac{ 1 }{ \sqrt{6} } + k c^{ 1/4 }( x + i y ) + h( c^{ 1/4 }( x + i y ) ),
\end{equation*}
where $ k $ is some coefficient and $ h( c^{ 1/4 }( x + i y ) ) $ satisfies an estimate of type (\ref{g_est}). Consequently, for $ \mathcal{F}( c, x, y ) $ we can write
\begin{equation*}
\mathcal{F}( c, x, y ) = f( 1 )( 6^{ -1/2 } + k c^{ 1/4 }( x + i y ) ) + 6^{ -1/4 } [ \partial_{ \zeta } f_{ D_+ }( 1 )( x + i y ) + \partial_{ \bar \zeta } f_{ D_+ }( 1 ) ( x - i y ) ] c^{ 1/4 } + \tilde g( c^{ 1/4 }( x + i y ) ),
\end{equation*}
where $ \tilde g( c^{ 1/4 }( x + i y ) ) $ satisfies an estimate of type (\ref{g_est}). This allows us to obtain in the end
\begin{equation}
\label{alpha_1_estimate}
\alpha_1( -18 t, t ) = \frac{ l_1( f( \pm 1 ) ) }{ t^{ 1/2 } } \left( 1 + \frac{ \const }{ t^{ 1/4 } } \right) + \frac{ l_2( f_{ \zeta }( \pm 1 ), f_{ \bar \zeta }(\pm 1 ) ) }{ t^{ 3/4 } } + o\left( \frac{ 1 }{ t^{ 3/4 } } \right), \quad t \to \infty,
\end{equation}
where $ l_1( f( \pm 1 ) ) $ is a linear combination of the limit values of $ f $ as $ \zeta $ tends to $ 1 $ and $ -1 $ from inside and outside of the unit circle, and $ l_2( f_{ \zeta }( \pm 1 ), f_{ \bar \zeta }(\pm 1 ) ) $ is a linear combination of the limit values of $ f_{ \zeta } $ and $ f_{ \bar \zeta } $ as $ \zeta $ tends to $ 1 $ and $ -1 $ from inside and outside of the unit circle.

Now let us consider the potential $ v_{ \theta } $ corresponding to the scattering data $ \theta b $, where $ \theta \in \mathbb{R} $ is some small parameter. In a way similar to which (\ref{alpha_1_estimate}) was obtained it can be shown that
\begin{equation*}
| f( \pm 1, t ) | \leqslant \frac{ c_1 \theta^2 }{ t^{ 1/2 } }, \quad | f_{ \zeta }( \pm 1, t ) | \leqslant c_2 \theta^2, \quad | f_{ \bar \zeta }( \pm 1, t ) | \leqslant c_2 \theta^2
\end{equation*}
for sufficiently large values of $ t $, where $ c_1 $, $ c_2 $, $ c_3 $ are some constants independent of $ t $ and $ \theta $.
% Substituting this into (\ref{alpha_1_estimate}) yields that $ \alpha_1( -18 t, t ) =  \frac{ l_2( f_{ \zeta }( \pm 1 ), f_{ \bar \zeta }(\pm 1 ) ) }{ t^{ 3/4 } } + o\left( \frac{ 1 }{ t^{ 3/4 } } \right) $ as $ t \to \infty $.
When $ \theta \to 0 $ and $ t \to \infty $, the linear approximation of $ v_{ \theta } $ behaves as $ O\left( \frac{ \theta }{ t^{ 3/4 } } \right) $, while the expression $ \alpha_1( -18t, t ) $ behaves as $ O\left( \frac{ \theta^2 }{ t^{ 3/4 } } \right) $ (it can be shown that the member $ o\left( \frac{ 1 }{ t^{ 3/4 } } \right) $ in (\ref{alpha_1_estimate}) depends quadratically on $ \theta $).

Finally, from (\ref{nu_representation}) and Lemma \ref{estimates_lemma} it follows that for $ \theta $ small enough and for $ z = -18 t $,
\begin{equation*}
v_{ \theta }( z, t ) = \frac{ C_{ \theta } }{ t^{ 3/4 } } + o\left( \frac{ 1 }{ t^{ 3/4 } } \right), \quad t \to \infty,
\end{equation*}
where $ C_{ \theta } $ is some nonzero constant. Thus we have shown that the estimate (\ref{main_estimate}) is optimal.

\section{Proofs of Lemma \ref{lin_estimate_lemma} and Lemma \ref{estimates_lemma}}
\label{proof_section}

\begin{proof}[Proof of Lemma \ref{lin_estimate_lemma}]
The proof follows the scheme described in Section \ref{lin_section} and is carried out separately for four cases depending on the values of the parameter $ u $. In all the reasonings that follow we denote by $ D_{ \varepsilon } $ the union of disks with the radius $ \varepsilon $ centered in the stationary points of $ S( u, \zeta ) $ and we denote by $ T $ the unit circle on the complex plane:
\begin{equation}
\label{t_definition2}
T = \{ \lambda \in \mathbb{C} \colon | \lambda | = 1 \};
\end{equation}
in addition, $ \const $ will denote an independent constant and $ \const( f ) $ will denote a constant depending only on function $ f $.

Case 1. $ u \in \mathbb{U} $

In this case all the stationary points belong to $ T $ and due to assumptions (\ref{f_assumptions}) and (\ref{zero_on_the_boundary}) of Lemma \ref{lin_estimate_lemma} we can estimate
\begin{equation}
\label{f_estimate}
| f( \zeta ) | \leqslant \const( f ) \varepsilon \text{ for } \zeta \in D_{ \varepsilon }.
\end{equation}
Now we estimate the integral $ I_{int} $ (as in (\ref{int_sum})) as follows
\begin{equation*}
| I_{ int } | = \left| \iint\limits_{ D_{ \varepsilon } } f( \zeta ) \exp( t S( u, \zeta ) ) d \Re \zeta d \Im \zeta \right| \leqslant \const( f ) \cdot \varepsilon \iint\limits_{ D_{ \varepsilon } } d \Re \zeta d \Im \zeta \leqslant \const( f ) \varepsilon^3.
\end{equation*}

The estimate for $ I_{ ext } $ (as in (\ref{int_sum})) is proved as follows.

We note that the function $ S'_{ \zeta }( u, \zeta ) $ can be estimated as
\begin{equation}
\label{der_estimate1}
\begin{aligned}
& | S'_{ \zeta }( u, \zeta ) | \geqslant \const \frac{ \varepsilon_0^3 }{ | \zeta |^4 } \quad \text{for} \: \zeta \in \mathbb{C} \backslash D_{ \varepsilon_0 }, \quad \text{and} \\
& | S'_{ \zeta }( u, \zeta ) | \geqslant \const \frac{ \rho^3 }{ | \zeta |^4 } \quad \text{for} \: \zeta \in \partial D_{ \rho }, \quad \varepsilon \leqslant \rho \leqslant \varepsilon_0.
\end{aligned}
\end{equation}

Similarly, we can estimate
\begin{equation}
\label{der_estimate2}
\begin{aligned}
& \left| \frac{ S''_{ \zeta \zeta }( u, \zeta ) }{ ( S'_{ \zeta }( u, \zeta ) )^2 } \right| \leqslant \const \frac{ | \zeta |^4 }{ \varepsilon_0^4 } \quad \text{for} \: \zeta \in \mathbb{C} \backslash D_{ \varepsilon_0 }, \quad \text{and} \\
& \left| \frac{ S''_{ \zeta \zeta }( u, \zeta ) }{ ( S'_{ \zeta }( u, \zeta ) )^2 } \right| \leqslant \const \frac{ | \zeta |^4 }{ \rho_0^4 } \quad \text{for} \: \zeta \in \partial D_{ \rho }, \quad \varepsilon \leqslant \rho \leqslant \varepsilon_0.
\end{aligned}
\end{equation}

Thus we obtain the following estimate for $ I_1 $ from (\ref{ext_by_parts})
\begin{equation*}
| I_1 | \leqslant \frac{ 1 }{ 2 } \int\limits_{ \partial D_{ \varepsilon } } \frac{ | f( \zeta ) | }{ | S'_{ \zeta }( u, \zeta ) | } | d \bar \zeta | \leqslant \const \frac{ \varepsilon }{ \varepsilon^3 } \int\limits_{ \partial D_{ \varepsilon } } | \zeta |^4 | d \bar \zeta | \leqslant \const( f ) \frac{ \varepsilon }{ \varepsilon^2 } ( 1 + \varepsilon )^4 \leqslant \frac{ \const( f ) }{ \varepsilon }.
\end{equation*}

Due to assumption (\ref{zero_on_the_boundary}) of Lemma \ref{lin_estimate_lemma} the integral $ I_2 $ from (\ref{ext_by_parts}) is equivalent to zero.

When estimating $ I_3 $ and $ I_4 $ from (\ref{ext_by_parts}) we fix some independent $ \varepsilon_0 > 0 $ and integrate separately over $ D_{ \varepsilon_0 } \backslash D_{ \varepsilon } $ and $ \mathbb{C} \backslash D_{ \varepsilon_0 } $:
\begin{multline*}
| I_3 | \leqslant \iint\limits_{ D_{ \varepsilon_0 } \backslash D_{ \varepsilon } } \left| \frac{ f'_{ \zeta }( \zeta ) \exp( t S( u, \zeta ) ) }{ S'_{ \zeta }( u, \zeta ) } \right| d \Re \zeta d \Im \zeta + \iint\limits_{ \mathbb{C} \backslash D_{ \varepsilon_0 } } \left| \frac{ f'_{ \zeta }( \zeta ) \exp( t S( u, \zeta ) ) }{ S'_{ \zeta }( u, \zeta ) } \right| d \Re \zeta d \Im \zeta \leqslant \\
\leqslant \const( f ) \int\limits_{ \varepsilon }^{ \varepsilon_0 } \frac{ \rho }{ \rho^3 } d \rho + \const \iint\limits_{ \mathbb{C} \backslash D_{ \varepsilon_0 } } | f'_{ \zeta }( \zeta ) || \zeta^4 | d \Re \zeta d \Im \zeta \leqslant \frac{ \const( f ) }{ \varepsilon },
\end{multline*}
\begin{multline*}
| I_4 | \leqslant \iint\limits_{ D_{ \varepsilon_0 } \backslash D_{ \varepsilon } } \left| \frac{ f( \zeta ) \exp( t S( u, \zeta ) ) S''_{ \zeta \zeta }( u, \zeta ) }{ ( S'_{ \zeta }( u, \zeta ) )^2 } \right| d \Re \zeta d \Im \zeta  + \\
+ \iint\limits_{ \mathbb{C} \backslash D_{ \varepsilon_0 } } \left| \frac{ f( \zeta ) \exp( t S( u, \zeta ) ) S''_{ \zeta \zeta }( u,\zeta ) }{ ( S'_{ \zeta }( u, \zeta ) )^2 } \right| d \Re \zeta d \Im \zeta \leqslant \\
\leqslant \const( f ) \int\limits_{ \varepsilon }^{ \varepsilon_0 } \frac{ \rho^2 }{ \rho^{4} } d \rho + \const \iint\limits_{ \mathbb{C} \backslash D_{ \varepsilon_0 } } | f( \zeta ) || \zeta^3 | d \Re \zeta d \Im \zeta \leqslant \frac{ \const( f ) }{ \varepsilon }.
\end{multline*}

Setting finally $ \varepsilon = \frac{ 1 }{ ( 1 + | t | )^{ 1 / 4 } } $ yields
\begin{equation*}
I( t, u ) \leqslant \frac{ \const( f ) }{ ( 1 + | t | )^{ 3/4 } }
\end{equation*}
uniformly on $ u \in \mathbb{U} $.

Case 2. $ u \in \mathbb{C} \backslash \mathbb{U} $ and $ \omega $ from (\ref{introduce_omega}) satisfies $ \omega_0 < \frac{ \omega }{ 1 + \omega } < 1 - \omega_1 $ for some fixed independent positive constants $ \omega_0 $ and $ \omega_1 $ (i.e. the roots $ \zeta_0 $,  $ \zeta_2 $ from (\ref{introduce_omega}) are separated from $ T $, defined by (\ref{t_definition2}), and the root $ \zeta_2 $ is separated from the origin)

In this case the we can estimate
\begin{align*}
& | S'_{ \zeta }( u, \zeta ) | \geqslant \frac{ \const \rho }{ | \zeta |^4 } \text{ for } \zeta \in \partial D_{ \rho }, \\
& \left| \frac{ S_{ \zeta \zeta }'' ( u, \zeta ) }{ ( S_{ \zeta }'( u, \zeta ) )^2 } \right| \leqslant \frac{ \const | \zeta |^4 }{ \rho^2 } \text{ for } \zeta \in \partial D_{ \rho }.
\end{align*}
Using these estimates and proceeding as in case 1, we obtain
\begin{equation*}
| I_{ int } | \leqslant \const( f ) \varepsilon^2, \quad | I_1 | \leqslant \const( f ), \quad I_2 \equiv 0, \quad | I_3 | \leqslant \const( f ), \quad | I_4 | \leqslant \const( f ) \ln \frac{ 1 }{ \varepsilon }.
\end{equation*}
Setting $ \varepsilon = \frac{ 1 }{ 1 + | t | } $, we obtain that
\begin{equation*}
I( t, u ) \leqslant \const( f ) \frac{ \ln( 3 + | t | ) }{ 1 + | t | }
\end{equation*}
uniformly for the considered values of the parameter $ u $.

Case 3. $ u \in \mathbb{C} \backslash \mathbb{U} $ and $ \frac{ \omega }{ 1 + \omega } < \omega_0 $ (i.e. the roots $ \zeta_0 $ and $ \zeta_2 $ from (\ref{introduce_omega}) lie in some neighborhood of $ T $ from (\ref{t_definition2}))

\begin{lemma}
\label{small_lemma}
For any $ t \geqslant t_0 $ with some fixed $ t_0 > 0 $ and any $ \omega > 0 $ one of the following conditions holds
\begin{itemize}
\item[(a)] $ 0 < \omega < \frac{ 2 }{ ( 1 + | t | )^{ 1/4 } } $;
\item[(b)] $ \omega > \frac{ 1 }{ ( 1 + | t | )^{ 1/8 } } $;
\item[(c)] $ \exists n \colon \frac{ 1 }{ ( 1 + | t | )^{ \gamma_{ n + 1 } / ( 2 + 2 \gamma_{ n + 1 } ) } } < \omega < \frac{ 2 }{ ( 1 + | t | )^{ \gamma_{ n } / ( 2 + 2 \gamma_{ n + 1 } ) } } $, where $ \gamma_{ n + 1 } = \frac{ 2 }{ 3 } \gamma_{ n } + \frac{ 1 }{ 3 } $, $ \gamma_1 = \frac{ 1 }{ 3 } $.
\end{itemize}
\end{lemma}
\begin{proof}
We note that
\begin{align*}
& \frac{ \gamma_{ n + 1 } }{ 2 + 2 \gamma_{ n + 1 } } \to \frac{ 1 }{ 4 }, \quad n \to \infty; \\
& \frac{ \gamma_n }{ 2 + 2 \gamma_{ n + 1 } } < \frac{ \gamma_n }{ 2 + 2 \gamma_n }; \\
& \frac{ \gamma_1 }{ 2 + 2 \gamma_2 } < \frac{ 1 }{ 8 }.
\end{align*}
Thus the intervals from the cases (a), (b), (c) $ \forall n \in \mathbb{N} $ cover the whole range $ 0 < \omega < +\infty $.
\end{proof}

We will prove the result separately for three different cases depending on the value of parameter $ \omega $
\begin{itemize}
\item[(a)] $ 0 < \omega < 2 \varepsilon = \frac{ 2 }{ ( 1 + | t | )^{ 1/4 } } $

In this case estimates (\ref{f_estimate}), (\ref{der_estimate1}), (\ref{der_estimate2}) hold and so the reasoning of the case 1 can be carried out to obtain that \begin{equation*}
I( t, u ) \leqslant \frac{ \const( f ) }{ ( 1 + | t | )^{ 3/4 } }
\end{equation*}
uniformly for the considered values of the parameter $ u $ satisfying
\begin{equation}
\label{left_range}
0 < \omega < \frac{ 2 }{ ( 1 + | t | )^{ 1/4 } }.
\end{equation}

\item[(b)] $ \omega > \varepsilon^{ 1/3 } = \frac{ 1 }{ ( 1 + | t | )^{ 1/8 } } $

In this case we estimate $ | I_{ int } | \leqslant \const( f ) \varepsilon^2 $.

Further, we note that the derivative of the phase is estimated as
\begin{equation}
\label{s_der_estimate}
| S'_{ \zeta }( u, \zeta ) | \geqslant \frac{ \const \, \varepsilon \omega^2 }{ | \zeta |^4 } \text{ for } \zeta \in \partial D_{ \varepsilon }.
\end{equation}
Thus for $ I_1 $ we obtain $ | I_1 | \leqslant \const( f ) \frac{ 1 }{ \varepsilon^{ 2/3 } } $.

In order to estimate the integral $ I_3 $ we use the following estimate of the derivative $ S'_{ \zeta } $ for $ \zeta \in \partial D_{ \rho } $ when $ \varepsilon \leqslant \rho \leqslant \varepsilon_0 $:
\begin{equation}
\label{s_der_multi_case}
\begin{cases}
& | S'_{ \zeta }( u, \zeta ) | \geqslant \frac{ \const \rho \omega^2 }{ | \zeta |^4 }, \text{ if } \rho < \omega, \\
& | S'_{ \zeta }( u, \zeta ) | \geqslant \frac{ \const \rho^3 }{ | \zeta |^4 }, \text{ if } \rho > \omega.
\end{cases}
\end{equation}
It allows to derive $ | I_3 | \leqslant \const( f ) \frac{ 1 }{ \varepsilon^{ 2/3 } } $.

Finally, we proceed to the study of the integral $ I_4 $. We use the following estimates
\begin{equation*}
\left| \frac{ S''_{ \zeta \zeta }( u, \zeta ) }{ ( S'_{ \zeta }( u, \zeta ) )^2 } \right| \leqslant \frac{ \const | \zeta |^4 }{ \rho^2 \omega^2 }
\end{equation*}
and
\begin{equation}
\label{f_gen_estimate}
\begin{cases}
& | f( \zeta ) | \leqslant \const( f ) \, \omega, \text{ if } \rho < \omega, \\
& | f( \zeta ) | \leqslant \const( f ) \, \rho, \text{ if } \rho > \omega.
\end{cases}
\end{equation}
After integration we obtain the estimate $ | I_4 | \leqslant \const( f ) \frac{ 1 }{ \varepsilon^{ 2/3 } } $.

Setting finally $ \varepsilon = \frac{ 1 }{ ( 1 + | t | )^{ 3/8 } } $, we obtain
\begin{equation*}
I( t, u ) \leqslant \frac{ \const( f ) }{ ( 1 + | t | )^{ 3/4 } }
\end{equation*}
uniformly for the considered values of the parameter $ u $ satisfying
\begin{equation}
\label{right_range}
\omega > \frac{ 1 }{ ( 1 + | t | )^{ 1/8 } }.
\end{equation}

\item[(c)] $ \varepsilon^{ \gamma_{ n + 1 } } < \omega < 2 \varepsilon^{ \gamma_{ n } } $, where $ \varepsilon = \frac{ 1 }{ ( 1 + | t | )^{ 1 / ( 2 + 2 \gamma_{ n + 1 } ) } } $ and $ \gamma_{ n + 1 } = \frac{ 2 }{ 3 } \gamma_n + \frac{ 1 }{ 3 } $, $ \gamma_1 = \frac{ 1 }{ 3 } $ (note that $ \gamma_n \to 1 $)

We proceed similarly to the case (b). Evidently, $ I_{ int } $ can be estimated $ | I_{ int } | \leqslant \const( f ) \varepsilon^{ 2 + \gamma_n } $. Employing the estimate (\ref{s_der_estimate}) we obtain $ | I_1 | \leqslant \const( f ) \frac{ \varepsilon^{ \gamma_n } }{ \varepsilon^{ 2 \gamma_{ n + 1 } } } $.

Using (\ref{s_der_multi_case}) in order to estimate $ I_3 $ we obtain $ | I_3 | \leqslant \const( f ) \frac{ 1 }{ \varepsilon^{ \gamma_{ n + 1 } } } $.

Finally, to estimate $ I_4 $ we use (\ref{f_gen_estimate}) and

\begin{equation*}
\left| \frac{ S''_{ \zeta \zeta } }{ ( S'_{ \zeta }( u, \zeta ) )^2 } \right| \leqslant \begin{cases}
& \frac{ \const | \zeta |^4 }{ \rho^2 \omega^2 }, \quad \rho < \omega, \\
& \frac{ \const | \zeta |^4 }{ \rho^3 \omega }, \quad \rho > \omega
\end{cases}
\end{equation*}
to obtain $ | I_4 | \leqslant \const( f ) \frac{ \varepsilon^{ \gamma_n } \ln( 1 / \varepsilon ) }{ \varepsilon^{ 2 \gamma_{ n + 1 } } } $.

Setting $ \varepsilon = \frac{ 1 }{ ( 1 + | t | )^{ 1 / ( 2 + 2 \gamma_{ n + 1 } ) } } $ yields
\begin{equation*}
| I( t, u ) | \leqslant \frac{ \const( f ) \ln( 3 + | t | ) ) }{ ( 1 + | t | )^{ 3/4 } }
\end{equation*}
uniformly for the considered values of the parameter $ u $ satisfying
\begin{equation}
\label{intermediate_range}
\frac{ 1 }{ ( 1 + | t | )^{ \gamma_{ n + 1 } / ( 2 + 2 \gamma_{ n + 1 } ) } } < \omega < \frac{ 2 }{ ( 1 + | t | )^{ \gamma_n / ( 2 + 2 \gamma_{ n + 1 } ) } }.
\end{equation}

\end{itemize}

Finally, from Lemma \ref{small_lemma} it follows that we have proved the required estimate uniformly on the values of parameter $ u \in \mathbb{C} \backslash \mathbb{U} $ such that $ \frac{ \omega }{ 1 + \omega } < \omega_0 $.

Case 4. $ u \in \mathbb{C} \backslash \mathbb{U} $ and $ \frac{ \omega }{ 1 + \omega } > 1 - \omega_1 $ (i.e. the roots $ \zeta_2 $, $ -\zeta_2 $ lie in the $ \omega_1 $--neighborhood of the origin)

This case is treated similarly to the previous one. We denote $ \tilde \omega = \frac{ 1 }{ 1 + \omega } $. Then we use estimates
\begin{equation*}
\begin{aligned}
& | f( \zeta ) | \leqslant c( f ) | \tilde \omega + \rho |, \\
& | S'_{ \zeta }( u, \zeta ) | \geqslant \frac{ \const \rho^2 }{ | \zeta |^4 } \text{ or } | S'_{ \zeta }( u, \zeta ) | \geqslant \frac{ \const \rho \tilde \omega }{ | \zeta |^4 }, \\
& \left| \frac{ S''_{ \zeta \zeta }( u, \zeta ) }{ ( S'_{ \zeta }( u, \zeta ) )^2 } \right| \leqslant \frac{ \const | \zeta |^4 }{ \rho^3 } \text{ or } \left| \frac{ S''_{ \zeta \zeta }( u, \zeta ) }{ ( S'_{ \zeta }( u, \zeta ) )^2 } \right| \leqslant \frac{ \const | \zeta |^4 }{ \rho^2 \tilde \omega },
\end{aligned}
\end{equation*}
which hold for $ \zeta \in D_{ \rho } $, to obtain the necessary estimates.
\end{proof}

\begin{proof}[Proof of Lemma \ref{estimates_lemma}]
\begin{enumerate}
\item The proof of inequality (\ref{beta1_estimate}) repeats the proof of Lemma \ref{lin_estimate_lemma}. The proof of inequality (\ref{beta2_estimate}) also follows the scheme of the proof of Lemma \ref{lin_estimate_lemma}. In this case we take $ D_{ \varepsilon } $ to be the union of disks of the radius $ \varepsilon $ with centers in the stationary points of $ S( u, \zeta ) $ and in the point $ \lambda $.

For the case when $ \lambda \not \in T $, where $ T $ is defined by (\ref{t_definition2}), an estimate weaker than (\ref{beta2_estimate}) can be obtained via a simplified reasoning. Indeed, $ I_{ int } $, as in (\ref{int_sum}), can be estimated $ | I_{ int } | \leqslant O\left( \frac{ \varepsilon }{ ( 1 + | t | )^{ \delta } } \right) $. Using estimates (\ref{der_estimate1}), (\ref{der_estimate2}) and
\begin{equation}
\label{lambda_estimate}
| \zeta - \lambda | \geqslant \rho \text{ for } \zeta \in \partial D_{ \rho }
\end{equation}
we obtain that $ | I_{ ext } | \leqslant O\left( \frac{ 1 }{ ( 1 + | t | )^{ 1 + \delta } \varepsilon^3 } \right) $. Setting $ \varepsilon = \frac{ 1 }{ ( 1 + | t | )^{ 1/4 } } $ we get the estimate (\ref{beta3_estimate}).

\item In order to obtain estimates (\ref{alpha1_estimate}), (\ref{alpha2_estimate}) we proceed according to the scheme outlined in Section \ref{lin_section}. In this case the integral $ I_2 $ does not annul. On the other hand, when the variable of integration belongs to $ T $, the estimate (\ref{beta2_estimate}) on the integrand of $ I_2 $ is stronger than the estimate (\ref{beta3_estimate}) for the general case. Thus we obtain for $ \alpha_1( z, t ) $
\begin{gather*}
| I_{ int } | \leqslant O\left( \frac{ \varepsilon^2 }{ ( 1 + | t | )^{ \delta + \frac{ 1 }{ 4 } } } \right), \quad | I_1 | \leqslant O\left( \frac{ 1 }{ ( 1 + | t | )^{ \delta + \frac{ 1 }{ 4 } } \, \varepsilon^2 } \right), \quad | I_2 | \leqslant O\left( \frac{ 1 }{ ( 1 + | t | )^{ \delta + \frac{ 1 }{ 4 } } \, \varepsilon^3 } \right), \\
| I_3 | \leqslant O\left( \frac{ 1 }{ ( 1 + | t | )^{ \delta } \, \varepsilon } \right), \quad | I_4 | \leqslant O\left( \frac{ 1 }{ ( 1 + | t | )^{ \delta + \frac{ 1 }{ 4 } } \, \varepsilon^2 } \right).
\end{gather*}
Setting $ \varepsilon = \frac{ 1 }{ ( 1 + | t | )^{ 1/4 } } $ yields the required estimate. The estimate (\ref{alpha2_estimate}) is obtained similarly.

\item We will give the scheme of the proof for estimate (\ref{gamma1_estimate}). The estimate (\ref{gamma2_estimate}) is obtained similarly.

We will prove (\ref{gamma1_estimate}) by induction. Suppose that (\ref{gamma1_estimate}) holds for all $ n = 1, 2, \ldots, N $. Then following the scheme of Section \ref{lin_section} and taking into account that $ \partial_{ \lambda } \overline{ ( A^n_{ z, t } \cdot f )( \lambda ) } = \overline{ ( A^{ n - 1 }_{ z, t } \cdot f )( \lambda ) } $, we obtain for $ n = N + 1 $:
\begin{gather*}
| I_{ int } | \leqslant O\left( \frac{ \varepsilon }{ ( 1 + | t | )^{ \delta + \frac{ 1 }{ 5 } \lceil \frac{ n - 1 }{ 2 } \rceil } } \right), \quad | I_1 | \leqslant O\left( \frac{ 1 }{ ( 1 + | t | )^{ \delta + \frac{ 1 }{ 5 } \lceil \frac{ n - 1 }{ 2 } \rceil } \varepsilon^3 } \right), \\
| I_2 | \leqslant O\left( \frac{ 1 }{ ( 1 + | t | )^{ \delta + \frac{ 1 }{ 5 } \lceil \frac{ n - 1 }{ 2 } \rceil } \varepsilon^4 } \right), \quad
| I_3 | \leqslant O\left( \frac{ 1 }{ ( 1 + | t | )^{ \delta + \frac{ 1 }{ 5 } \lceil \frac{ n - 2 }{ 2 } \rceil } \varepsilon^3 } \right), \\
| I_4 | \leqslant O\left( \frac{ 1 }{ ( 1 + | t | )^{ \delta + \frac{ 1 }{ 5 } \lceil \frac{ n - 1 }{ 2 } \rceil } \varepsilon^3 } \right).
\end{gather*}

Setting $ \varepsilon = \frac{ 1 }{ ( 1 + | t | )^{ 1/5 } } $ we obtain the required estimate.

\item We represent $ R_{ z, t }( \lambda ) $ as the sum of the following members
\begin{multline*}
R_{ z, t }( \lambda ) = B( A + A^2 + A^3 + \ldots ) \cdot 1 + A B ( A + A^2 + A^3 + \ldots ) \cdot 1 + \\
+ ( A + A^2 + A^3 + \ldots ) A B ( I + A + A^2 + \ldots ) \cdot 1 = R^1_{ z, t }( \lambda ) + R^2_{ z, t }( \lambda ) + R^3_{ z, t }( \lambda ).
\end{multline*}
The convergence of the series at sufficiently large times follows from the estimate (\ref{gamma2_estimate}).
Now let
\begin{equation*}
R^{ i }_{ z, t }( \lambda ) = \frac{ q_i( z, t ) }{ \lambda } + o\left( \frac{ 1 }{ | \lambda | } \right), \text{ as } \lambda \to \infty.
\end{equation*}
From (\ref{beta1_estimate}) and (\ref{gamma2_estimate}) it follows that $ | q_1( z, t ) | \leqslant \frac{ \hat q_1( c_f ) \ln( 3 + | t | ) }{ ( 1 + | t | )^{ 3/4 + 1/5 } } $. From (\ref{alpha1_estimate}) and (\ref{gamma2_estimate}) we obtain that $ | q_2( z, t ) | \leqslant \frac{ \hat q_2( c_f ) }{ ( 1 + | t | )^{ 3/4 + 1/5 } } $. Finally, from (\ref{alpha2_estimate}), (\ref{gamma1_estimate}) and (\ref{gamma2_estimate}) it follows that $ | q_3( z, t ) | \leqslant \frac{ \hat q_3( c_f ) }{ ( 1 + | t | )^{ 1/2 + 2/5 } } $. This yields the required estimate. \qedhere

\end{enumerate}
\end{proof}

\end{document}